\documentclass[11pt]{amsart}
\usepackage{fullpage, amsmath, amsthm,amsfonts,amssymb,stmaryrd, mathrsfs, mathtools}
\usepackage{hyperref}
\usepackage{mathdots}
\usepackage[pdftex]{color}
\usepackage[all]{xy}
\usepackage{xcolor}
\usepackage{hyperref}
\usepackage{bbm}
\usepackage{xspace}
\usepackage[T1]{fontenc}

\usepackage[autolanguage]{numprint}
\usepackage{relsize}
\usepackage[bbgreekl]{mathbbol}
\usepackage{tikz-cd}
\usepackage[lite,abbrev,msc-links,backrefs]{amsrefs}

\numberwithin{equation}{subsection}
\newtheorem{theorem}[subsection]{Theorem}

\newtheorem{corollary}[subsection]{Corollary}
\newtheorem{lemma}[subsection]{Lemma}
\newtheorem{proposition}[subsection]{Proposition}

\theoremstyle{definition}
\newtheorem{example}[subsection]{Example}
\newtheorem{definition}[subsection]{Definition}

\newtheorem{situation}[subsection]{Situation}
\newtheorem{assumption}[subsection]{Assumption}
\newtheorem{remark}[subsection]{Remark}

\DeclareSymbolFontAlphabet{\mathbb}{AMSb} %to ensure that the meaning of \mathbb does not change
\DeclareSymbolFontAlphabet{\mathbbl}{bbold}
\newcommand{\prism}{{\mathlarger{\mathbbl{\Delta}}}}

\def\bfP{\mathbf{P}}

\def\NN{\mathbb{N}}

\def\ZZ{\mathbb{Z}}

\def\scrC{\mathscr{C}}

\newcommand{\rd}{\mathrm{d}}

\newcommand{\rT}{\mathrm{T}}

\def\calC{\mathcal{C}}

\def\calE{\mathcal{E}}
\def\calF{\mathcal{F}}

\def\calI{\mathcal{I}}
\def\calJ{\mathcal{J}}
\def\calK{\mathcal{K}}

\def\calM{\mathcal{M}}

\DeclareMathOperator{\End}{End}

\DeclareMathOperator{\Hom}{Hom}

\DeclareMathOperator{\Bl}{Bl}
\DeclareMathOperator{\RHom}{R\mathcal{H}om}

\DeclareMathOperator{\Spec}{Spec}
\DeclareMathOperator{\Spf}{Spf}
\DeclareMathOperator{\Mod}{\mathbf{Mod}}

\DeclareMathOperator{\Sym}{Sym}
\DeclareMathOperator{\DR}{DR^{\bullet}}

\newcommand{\cO}{\mathcal{O}}

\newcommand{\pr}{\mathrm{pr}}

\newcommand{\CR}{\mathbf{CR}}
\newcommand{\CA}{\mathrm{\check{C}A}}

\newcommand{\ra}{\rightarrow}

\newcommand{\xra}{\xrightarrow}

\newcommand{\et}{\mathrm{\acute{e}t}}

\newcommand{\Higgs}{\mathbf{Higgs}}

\newcommand{\id}{\mathrm{id}}

\newcommand{\Kos}{\mathrm{Kos}}

\newcommand{\FMod}{\mathbf{FMod}}
\newcommand{\Desc}{\mathbf{Desc}^{\wedge}}
\newcommand{\FDesc}{\mathbf{FDesc}}

\newcommand{\Strat}{\mathbf{Strat}}
\newcommand{\ev}{\mathrm{ev}}
\newcommand{\mbar}{\underline{m}}
\newcommand{\Et}{\mathrm{\acute Et}}
\begin{document}

	\title{Finiteness and duality for the cohomology of prismatic crystals}
	
	\author{Yichao Tian}
	\address{Morningside Center of Mathematics,  Hua Loo-Keng Key Laboratory of Mathematics\\
		Academy of Mathematics and Systems Science, CAS\\
		55 Zhong Guan Cun East Road\\
		100190, Beijing, China}
	\email{yichaot@math.ac.cn}
	\date{}

	\begin{abstract}
		Let $(A, I)$ be a  bounded prism, and $X$ be a  smooth $p$-adic formal scheme over $\Spf(A/I)$.  We consider the notion of  crystals on  Bhatt--Scholze's prismatic site $(X/A)_{\prism}$ of $X$ relative to $A$. We prove that if $X$ is proper over $\Spf(A/I)$ of relative dimension $n$, then the cohomology of a prismatic crystal is a perfect complex of $A$-modules with tor-amplitude in degrees $[0,2n]$. We also establish a Poincar\'e duality for the reduced prismatic crystals, i.e. the crystals over  the reduced structural sheaf of $(X/A)_{\prism}$. The key ingredient is an explicit local description of reduced prismatic crystals in terms of Higgs modules. 
	\end{abstract}
	\maketitle
	%\tableofcontents
	
	\setcounter{section}{-1}
	
	\section{Introduction}
	%A fundamental and important  question in algebraic geometry is to study the relationship of various cohomological theories attached to a common variety. 
	%In their  recent work \cite{BMS1}, Bhatt, Morrow and Scholze defined a new cohomology theory, called $A_{\inf}$-cohomology, for a smooth $p$-adic formal scheme $X$ over  the ring of integers of the completion of $\QQ_p$. This new cohomology theory  can be  related to all known other $p$-adic cohomology theories of $X$, and it  provides new $p$-adic comparison theorems with integral coefficients. Their construction of the $A_{\inf}$-cohomology is a little ad-hod by using the magical $L\eta$-functor.
	In a recent ground breaking work \cite{BS}, Bhatt and Scholze introduced the prismatic site for $p$-adic formal schemes (with $p$  a fixed prime), and they studied  the cohomology of the natural structural sheaf on the prismatic site, called the prismatic cohomology.   This new cohomology theory seems to occupy a central role in the study of cohomological properties of $p$-adic formal schemes, since it is naturally related to various previously known $p$-adic cohomology theories, and thus provides new insight on the $p$-adic comparison theorems with integral coefficients. For instance, it gives a natural site theoretic construction of the $A_{\inf}$-cohomology which was previously constructed  by Bhatt--Morrow--Scholze  \cite{BMS1} in an ad-hoc way using the magical $L\eta$-functor. 
	
	As  analogues of  classical crystalline crystals, there is a natural notion of crystals on the  prismatic site. Indeed, such objects have already been considered in some special cases  by many authors such as Ansch\"utz--Le Bras  \cite{AB}, Gros--Le Strum--Quir\'os \cite{GSQ}, Li \cite{Li} and Morrow--Tsuji \cite{MS}. In this article, we will consider the cohomology of   prismatic crystals on rather general  prismatic sites and prove some finiteness and duality theorems for such crystals.

	Let us explain the main results of this article in more detail. Let $(A,I)$ be a  bounded prism, and $X$ be a smooth $p$-adic formal scheme over $A/I$. We denote by $(X/A)_{\prism}$ Bhatt--Scholze's prismatic site of $X$ relative to $A$ \cite{BS}*{Def. 4.1}. The category $(X/A)_{\prism}$ consists of   bounded prisms $(B,J)$ over $(A,I)$ together with a structural map $\Spf(B/J)\to X$, and coverings in $(X/A)_{\prism}$ are $(p,I)$-completely faithfully flat maps of such bounded prisms.  We have the  structural sheaf $\cO_{\prism}$ (resp. the reduced structural sheaf $\overline{\cO}_{\prism}$) which sends each object $(B,J)$ to $B$ (resp. to  $B/J$).
	An $\cO_{\prism}$-crystal (resp.  an $\overline{\cO}_{\prism}$-crystal) is a $(p,I)$-completely  flat  and derived $(p,I)$-complete $\cO_{\prism}$-module (resp. $\overline{\cO}_{\prism}$-module) that satisfies similar  properties as classical crystals on crystalline sites (cf. Def.~\ref{D:crystals}).  Here, we insist to impose the flatness condition in order to avoid some technical difficulties in faithfully flat descent for derived $(I, p)$-complete modules (cf. Prop.~\ref{P:flat-descent}).

	The first main result of this article is a finiteness theorem for $\cO_{{\prism}}$-crystals (Theorem~\ref{T:finiteness-crystal}), which claims that, if $X$ is proper and smooth of relative dimension $n$ over $A/I$ and $\calF$ is an $\cO_{\prism}$-crystal locally free of finite rank, then $R\Gamma((X/A)_{\prism},\calF)$ is a perfect complex of $A$-modules with tor-amplitude  in degree $[0,2n]$. Moreover, the formation of $R\Gamma((X/A)_{\prism},\calF)$ commutes with arbitrary base change in $A$.
	This finiteness theorem for $\cO_{{\prism}}$-crystals is a consequence of  a similar  result for $\overline{\cO}_{\prism}$-crystals. Actually, if $\calE$ is an $\overline{\cO}_{\prism}$-crystal locally free of finite rank, we will show  in Theorem~\ref{T:Obar-crystal} that the derived push-forward of $\calE$ to the \'etale topos of $X$ is a perfect complex of $\cO_X$-modules with tor-amplitude  in $[0,n]$. 
	
	We   are  thus reduced  to the study of $\overline{\cO}_{\prism}$-crystals. Since the problem is local for the \'etale topology of $X$, we may assume that $X=\Spf(R)$ is affine such that    $R$ is  $p$-completely \'etale   over the convergent power series ring $A/I\langle T_1,\dots, T_n\rangle$. In this case, we can give a rather explicit description of $\cO_{{\prism}}$-crystals in terms of Higgs modules. More precisely,  
	after choosing a smooth lift $\widetilde R$ over $A$ of $R$ together with a $\delta$-structure on $\widetilde R$ compatible with that on $A$, we can show that there exists an equivalence  between the category of $\overline{\cO}_{{\prism}}$-crystals and that of topologically quasi-nilpotent Higgs modules over $R$ (cf. Theorem~\ref{T:crystals-Higgs});
	furthermore, the cohomology of an $\overline{\cO}_{{\prism}}$-crystal is computed by  the de Rham complex of its associated Higgs module (cf. Theorem~\ref{T:cohomology-Higgs}). From this description, our  finiteness theorem for $\overline{\cO}_{{\prism}}$-crystals  follows easily.
	
	As another application of the local description of $\overline{\cO}_{{\prism}}$-crystals, we establish also in Theorem~\ref{T:poincare-pairing} a Poincar\'e duality for the cohomology of $\overline{\cO}_{{\prism}}$-crystals. This  can be viewed as a combination of the duality for de Rham complexes of Higgs modules and the classical Grothendieck--Serre duality. If one can construct a trace map for the prismatic cohomology of proper and smooth formal schemes, our results imply also a  Poincar\'e duality for $\cO_{{\prism}}$-crystals as well (cf. Remark~\ref{R:poincare-duality}).

	The organization of this article is as follows. In Section~1, we prove some preliminary results in commutative  algebra. The main result of this section is a descent result on derived $I$-complete and $I$-completely flat modules (cf. Proposition~\ref{P:flat-descent}). In Section 2, we discuss the notion of prismatic crystals and state the main finiteness theorems. Section 3 is devoted to the local study of prismatic crystals. When $X=\Spf(R)$ is affine equipped with a lifting $\widetilde R$ over $A$ together with a $\delta$-structure, we show that the category of prismatic crystals is equivalent to that of modules over $\widetilde R$ equipped with a certain stratification (cf. Proposition \ref{P:crystal-stratification}). In Section 4, we work in the local situation of affine formal schemes equipped with \'etale local coordinates, and we related $\overline{\cO}_{\prism}$-crystals to Higgs modules as mentioned above. Then we finish the proof of finiteness theorems at the end of Section 4. Finally, in Section~\ref{S:global application}, we prove a Poincar\'e duality for $\overline{\cO}_{{\prism}}$.

	It should be pointed out   that  Frobenius structures and  filtrations on prismatic crystals are ignored in this article, even though these aspects should play important roles in many applications.
	
	The results of this article was announced in a conference in honor of Luc Illusie in June 2021. At this conference, I learned that similar  results in this article were also obtained independently by Ogus \cite{Ogus} for crystalline prisms  and Bhatt--Lurie \cite{BL} for absolute prismatic crystals.

	\subsection*{Acknowledgement}
	The author would like to thank Arthur Ogus, Bernard Le Strum for helpful discussion on the contents of the paper, and  the anonymous referee for many valuable comments.
	This work is  supported by CAS Project for Young Scientists in Basic Research  (Grant No. YSBR-033) and Chinese National Science Fund for Distinguished Young Scholars (Grant No. 12225112).  
	
	%We hope to pursuit these aspects in   
	\subsection{Notation}
	Let $A$ be a commutative ring. We denote by $\Mod(A)$ the abelian  category of $A$-modules, and by $D(A)$ the derived category of $\Mod(A)$.
	If $M$ is an $A$-module and $f\in A$ is an element, we denote by $M[f]$ the kernel of the multiplication by $f$ on $M$. We put also $M[f^{\infty}]=\bigcup_{n\geqslant 1}M[f^n]$. 
	
	Let  $J\subset A$ be a PD-ideal. For $x\in J$ and an integer $n\geqslant 1$, we denote by $x^{[n]}$ the $n$-th divided power of $x$.
	
	\subsection{Sign Conventions}\label{S:sign-covention}
	We will use the following conventions on the signs of complexes:
	
	\begin{itemize}
		\item For a na\"ive double complex $K^{\bullet, \bullet}$, its associated simple complex  $\underline{s}(K^{\bullet, \bullet})$  has $n$-th differential  given by 
		\[
		\sum_{i+j=n}(d_{1}^{i,j}+(-1)^id_{2}^{i,j})\colon	\bigoplus_{i+j=n}K^{i,j}\to \bigoplus_{a+b=n+1} K^{a,b},
		\]
		where $d_1^{i,j}: K^{i,j}\to K^{i+1,j}$ and $d_2^{i,j}\colon K^{i,j}\to K^{i,j+1}$ are the two differentials of $K^{\bullet,\bullet}$ satisfying $d_1^{i,j+1}\circ d_2^{i,j}=d_{2}^{i+1, j}\circ d_{1}^{i,j}$. 
		
		\item For a complex $K$ and an integer $n\in \ZZ$, the $i$-th differential of the shift $K[n]$ is obtained by multiplying $(-1)^{n}$ with the $(i+n)$-th differential of $K$.
		
		\item For two complexes $K$ and $L$, their tensor product $K\otimes L$ is the simple complex attached  to the na\"ive double complex with $(i,j)$-component given by $K^i\otimes L^j$ (whenever such a tensor product is well defined). 
		The hom-complex $\Hom^\bullet(K,L)$ is defined as the  simple complex with $n$-th term
		\[
		\Hom^n(K, L)=\prod_{i+j=n}\Hom(K^{-j}, L^i)
		\]
		and $n$-th differential 
		\[
		d(f)=d_L\circ f-(-1)^nf\circ d_K
		\]

	\end{itemize}

	\section{Preliminaries in commutative algebra}

	%In this section, we prove some preliminary results on derived completion and $I$-completely flat modules. 

	We recall first some general facts on derived completion, for which  the main reference is \cite{stacks-project}*{\href{https://stacks.math.columbia.edu/tag/091N}{Tag 091N}}.
	
	We consider a pair $(A,J)$, where $A$ is a commutative ring, and $J\subset A$  is an ideal.   A complex $K$ of $A$-modules is called derived $J$-complete, if $R\Hom(A_f,K)=0$ for all $f\in J$, where $A_f=A[\frac{1}{f}]$; and an $A$-module $M$ is called derived $J$-complete  if $M[0]$ is derived $J$-complete. 
	Then $K$ is derived $J$-complete if and only if so are $H^q(K)$ for all $q\in \ZZ$.
	The derived $J$-complete $A$-modules form an abelian full subcategory of $\Mod(A)$ that is stable under kernels, cokernels, images and extensions. Any classically $J$-adically complete $A$-module is derived $J$-complete, but the converse is not necessarily true. 
	Let    $D_{comp}(A)$ be  the full subcategory   consisting of $D(A)$ consisting of derived $J$-complete objects. 
	
	Assume from now on that $J$ is finitely generated. Then the natural inclusion functor $D_{comp}(A)\to D(A)$ admits a left adjoint $K\mapsto \widehat{K}$, called the \emph{derived $J$-completion}. An explicit construction of $\widehat{K}$ is given as follows. Write $J=(f_1,\dots, f_r)$. Let $\Kos(A; f_1,\dots, f_r)$ be the homological Koszul complex sitting in degrees $[-r, 0]$: 
	\[ \wedge ^r A^{\oplus r}\to \wedge^{r-1}A^{\oplus r}\to\cdots \to A^{\oplus r}\xra{(f_1,\dots, f_r)} A.\]
	For an integer $n\geqslant 1$, we have a transition map of complexes \[\Kos(A; f_1^{n+1},\dots, f_r^{n+1})\to \Kos(A; f_1^{n},\dots,f_r^{n} )\]
	given by the multiplication by  $f_{i_1}\cdots f_{i_m}$ on a basis element $$e_{i_1}\wedge \cdots\wedge e_{i_m}\in \Kos^{-m}(A; f_1^{n+1},\dots,f_r^{n+1} )=\wedge^m A^{\oplus r}.$$ Then for an object $K\in D(A)$, we have 
	\begin{equation}\label{E:derived completion}
		\widehat{K}=R\lim_n \big(K\otimes_{A}^L \Kos(A; f_1^n,\dots, f_r^n)\big). \end{equation}
	%By adjointness, any morphism $K\to L$ in $D(A)$ with $L\in D_{comp}(A)$ canonically factori
	Note that, in general, the canonical map  $\widehat{K}\to R\lim_n(K\otimes_A^LA/J^n)$ is not an isomorphism in $D(A)$. However, we will see later (Prop. \ref{P:completely-flat}) that  this indeed holds  in an important special case.

	\begin{lemma}\label{L:completion}
		For an object $K$ of $D(A)$ and an integer $n\geqslant 1$, the canonical maps  
		\begin{align*}K\otimes^L\Kos(A;f_1^{n},\dots, f_r^n)&\xra{\sim}\widehat{K}\otimes^L\Kos(A;f_1^{n},\dots, f_r^n), \\ 
			K\otimes^L_AA/J^n&\xra{\sim}\widehat{K}\otimes_A^LA/J^n\end{align*}
		are  isomorphisms.
	\end{lemma}
	\begin{proof}
		Since the cohomology groups of $\Kos(A; f_1^n,\dots,f_r^n )$ are annihilated by $(f_1^n,\cdots, f_r^n)$ and hence by $J^{(n-1)r+1}$, it follows from  \cite{stacks-project}*{\href{https://stacks.math.columbia.edu/tag/091W}{Tag 091W}} that $K\otimes^L_A\Kos(A; f_1^n,\dots,f_r^n )$ is already derived $J$-complete. Hence, we have 
		\[K\otimes^L_A\Kos(A; f_1^n,\dots,f_r^n )=(K\otimes^L_A\Kos(A; f_1^n,\dots,f_r^n ))^{\wedge}\simeq \widehat{K}\otimes^L_A\Kos(A; f_1^n,\dots,f_r^n ).\]
		The Lemma follows by further base change via the canonical map $\Kos(A; f_1^n,\dots,f_r^n )\to A/J^n$.
	\end{proof}
	
	A complex of $A$-modules $K$ is called \emph{$J$-completely flat} (resp. \emph{$J$-completely faithfully flat}) if $K\otimes^L_AA/J$ is concentrated in degree $0$ and  is a flat (resp. faithfully flat) $A/J$-module. An object $K\in D(A)$ is called \emph{$J$-completely locally free of finite type} if it is $J$-completely flat and $K\otimes_A^L A/J$ is a  locally free $A/J$-module of finite type.  It is clear that $J$-complete flatness is stable under (derived) base change in $(A,J)$, and  by   Lemma~\ref{L:completion} an object $K\in D(A)$ is $J$-completely flat if and only if  so is $\widehat{K}$.

	\begin{lemma}\label{L:flat-local}
		Let $A\to B$ be a $J$-completely faithfully flat map of rings.

		\begin{enumerate}
			\item An object  $K$ of $D(A)$ is $J$-completely flat (or $J$-completely locally free of finite type) if and only if so is $K_B:=(K\otimes^L_AB)^{\wedge}$.
			
			\item Let  $\phi: M\to N$ be a morphism in $D(A)$ with $M,N$ derived  $J$-complete. Then $\phi$ is a  quasi-isomorphism if and only if so is $\phi_B: M_B\to N_B$.
		\end{enumerate}
	\end{lemma}

	\begin{proof}
		(1)	By Lemma~\ref{L:completion}, we have    
		\begin{equation}\label{E:flat base change} K_B\otimes^L_BB/JB=K\otimes^L_A B/JB=(K\otimes_A^LA/J)\otimes_{A/J}^L B/JB.\end{equation}
		Statement (1) follows immediately from the usual fpqc-descent of  modules via the faithfully flat map  $A/J\to B/JB$.
		
		(2) Pick a distinguished triangle $M\xra{\phi}N \to K\to M[1]$ in $D(A)$, from which we deduce a distinguished triangle in $D(B)$: 
		$M_B\xra{\phi_B} N_B\to K_B\to M_B[1].$ We need to prove that $K=0$ if and only if $K_B=0$. In view of  \eqref{E:flat base change} and the faithful flatness of $B/JB$ over $A/J$, one has $K\otimes_A^LA/J=0$ if and only if $K_B\otimes^L_BB/JB=0$. Then we conclude by the derived Nakayama Lemma \cite[\href{https://stacks.math.columbia.edu/tag/0G1U}{Tag 0G1U}]{stacks-project}.
	\end{proof}

	\if false

	\begin{lemma}\label{L:bounded-torsion}
		Let $f\in A$.  Assume that $A$ has bounded $f^{\infty}$-torsion, i.e. there exists an integer $c\geqslant 1 $ such that $A[f^{\infty}]=A[f^c]$. Let $N$ be a classically $f$-adically complete $A$-module such that $N/f^nN$ is a (faithfully) flat $A/f^nA$-module for any integer $n\geqslant 1$. Then $N$ is a  $f$-completely (faithfully) flat and derived $f$-complete $A$-module.
	\end{lemma}
	\begin{proof}
		By the same proof as \cite[Lemma 4.7]{BMS2} (with $p$ replaced by $f$), it suffices to show that $M$ has bounded $f^{\infty}$-torsion. By assumption, there exists an integer $c\geqslant 1$ such that $A[f^{\infty}]=A[f^c]$. Put $A'=A/A[f^c]$, which is $f$-torsion free. Put  $N'=N\otimes_AA'$. Then by the exact sequence
		\[A[f^c]\otimes_A N\to N\to N'\to 0,\] it is enough to show that $N'$ is $f$-torsion free. For any integer $n\geqslant 1 $, let $A_n'=A'/f^nA'$ and $N'_n=N'/f^nN'$.   By the $f$-torsion freeness of $A'$, we have an exact sequence 
		\[0\to A'_{n-1}\xra{\times f}A'_n\to A'_1\to 0.\]
		Note that the flatness of $N/f^nN$ over $A/f^nA$ implies the flatness of  $N'_n$ over $A'_n$. Tensoring with $N'_n$, we get  
		\[0\to N'_{n-1}\xra{\times f} N'_n\to N'_1\to 0.\]
		Now let $x\in N'$ with $fx=0$. Then the above exact sequence implies that $x\in f^{n-1}N'$ for any integer $n\geqslant1 $. But $N'$ is $f$-adically complete, we get thus  $x=0$. This finishes the proof of the Lemma.
	\end{proof}

	\fi
	The following definition is motivated by the notion of bounded prism \cite[Def. 3.2]{BS}. 
	\begin{definition}\label{D:prism-type}
		
		A  pair $(A, J)$ is of  \emph{reduced prismatic type } if $J=(f)$ for some $f\in A$,  $A$ is derived $J$-complete and has bounded $f^{\infty}$-torsion, i.e. there exists an integer $c\geqslant 0$ such that $A[f^{c}]=A[f^{\infty}]$.

		A pair $(A,J)$ is of  \emph{prismatic type} if $A$ is derived $J$-complete, where  $J\subset A$ is an ideal of the form $J=(I, f)$ such that 
		\begin{itemize} 
			\item $I\subset A$ is a finitely generated ideal that defines a Cartier divisor of $\Spec(A)$,
			\item $I^r$ is a  principal ideal for some integer $r\geqslant 1$,
			\item and $f\in A$ is  a nonzero divisor  such that $A/I$ has bounded $f^{\infty}$-torsion.
		%	\item $J=(I, f)$ and $A$ is derived $J$-complete.
		\end{itemize}

	\end{definition}

\begin{remark}\label{R:prismatic-reduction}
By \cite[Lemma  3.6]{BS}, a bounded prism   is a pair of prismatic type in the  sense above. If $(A,J)$ is a pair of prismatic type with $J=(I, f)$ as in the definition, then $(A,J_r)$ with $J_r=(I^r,f)$ is also a pair of  prismatic type.  Moreover,  as $J^r\subset J_r\subset J$,  the derived  completion and the classical adic completion with respect to  $J$ agree with those with respect to   $J_r$ (see \cite[\href{https://stacks.math.columbia.edu/tag/091S}{Tag 091S}]{stacks-project}). This allows us to reduce many arguments below to the case of principal  $I$, i.e. $J=(\xi, f)$ for  two nonzero divisors $\xi, f\in A$ such that    $A/(\xi)$ has  bounded $f^{\infty}$-torsion. 
	\end{remark}

	\begin{proposition}\label{P:completely-flat}
		Let  $(A,J)$ ba  a pair of (reduced)  prismatic type. 
		
		\begin{enumerate}
			\item For an object $K\in D(A)$, the canonical map 
			\[\widehat{K}\xra{\sim}R\lim_n (K\otimes_{A}^L A/J^n)\]
			deduced from the universal property  of $\widehat{K}$ is an isomorphism.
			
			\item If $M$ is a $J$-completely (faithfully) flat complex of $A$-modules, then $\widehat{M}$ is concentrated in degree $0$ and a classically $J$-adically complete $A$-module such that $\widehat{M}/J^n\widehat{M}$ is (faithfully) flat over $A/J^n$. In particular, $A$ is classically $J$-complete. %Moreover, for any integer $n\geqslant 1$, $\widehat{M}/\xi^n \widehat{M}$ has bounded $f^{\infty}$-power torsion.
			
			\item Conversely, if $N$ is a classically $J$-adically complete $A$-module such that $N/J^nN$ is (faithfully) flat over $A/J^n$, then $N$ is a $J$-completely (faithfully) flat and derived $J$-complete object in $D(A)$. 
		\end{enumerate}
	\end{proposition}
	\begin{proof}
		
		For (1), we will treat only the case when $(A,J)$ is a  pair of prismatic type, the arguments for the case of reduced prismatic type being  similar and much simpler. 
	As explained in Remark~\ref{R:prismatic-reduction}, we may assume that $J=(\xi, f)$ such that $\xi, f\in A$ are nonzero divisor and $A/\xi A$ has bounded $f^{\infty}$-torsion.
		
		The arguments follows from the same line as  \cite[Lemma 3.7(1)]{BS}. By \eqref{E:derived completion} and a simple cofinality argument, we have 
		\[\widehat{K}=R\lim_nR\lim_m (K\otimes^L_A\Kos(A; \xi^n,f^m)).\] 
		Since $\xi$ is a nonzero divisor, we have a quasi-isomorphism 
		\[\Kos(A; \xi^n,f^m)\simeq \Kos(A/\xi^n; f^m).\]
		By d\'evissage, the assumption that $A/\xi A$ has bounded $f^{\infty}$-torsion implies that so does $A/\xi^nA$, which in turn implies that the pro-object  $\{\Kos(A/\xi^nA; f^m):m\geqslant 1\}$ is isomorphic to $\{A/(\xi^n,f^m):m\geqslant 1 \}$ by  \cite[\href{https://stacks.math.columbia.edu/tag/091X}{Tag 091X}]{stacks-project}. Hence, we get   
		\begin{align*}
			\widehat{K}&=  R\lim_nR\lim_m (K\otimes^L_A\Kos(A; \xi^n,f^m))\\
			&\simeq R\lim_n R\lim_m (K\otimes^L_A\Kos(A/\xi^n; f^m))\\
			&\simeq R\lim_n R\lim_m (K\otimes_A^LA/(\xi^n,f^m))\\
			&\simeq R\lim_n(K\otimes^L_AA/J^n),
		\end{align*}
		where the last step follows from a cofinality argument.
		
		Statement (2) follows from the same argument as \cite[Lemma 3.7(2)]{BS}. 
		
		For (3), we consider first the case when $(A,J)$ is a pair of reduced prismatic type. We write thus $J=(f)$ with $f\in A$ such that $A$ has bounded $f^{\infty}$-torsion. By the same proof as \cite{BMS2}*{Lemma 4.7} (with $p$ replaced by $f$), it suffices to show that $N$ has bounded $f^{\infty}$-torsion. By assumption, there exists an integer $c\geqslant 1$ such that $A[f^{\infty}]=A[f^c]$. Put $A'=A/A[f^c]$, which is $f$-torsion free. Put  $N'=N\otimes_AA'$. 
		Then by the exact sequence
		\[A[f^c]\otimes_A N\to N\to N'\to 0,\] it is enough to show that $N'$ is $f$-torsion free.  We prove first $N'$ is  $f$-adically complete.
		For any integer $n\geqslant 1 $, let $A_n'=A'/f^nA'$, $N_n=N/f^nN$ and $N'_n=N'/f^nN'$.   
		As $A'$ is $f$-torsion free, we have an exact sequence 
		\[0\to A[f^{c}]\to A_n\to A_n'\to 0\]
		for each $n\geqslant c$. Tensoring with $N_n$ over $A_n$, we get 
		\[0\to A[f^c]\otimes_{A_n}N_n\to N_n\to N_n'\to 0\]
		where the left exactness follows from the flatness of $N_n$ over $A_n$. Note that  $ A[f^c]\otimes_{A_n}N_n\cong A[f^c]\otimes_AN$ is independent of $n$ for $n\geqslant c$. Taking inverse limits, one gets thus 
		\[0\to A[f^c]\otimes_AN\to N\to \varprojlim_nN_n'\to 0,\]
		and hence $N'=\varprojlim_n N'_n$ is complete. 	
		Now note that  the $f$-torsion freeness of $A'$ implies  an exact sequence 
		\[0\to A'_{n-1}\xra{\times f}A'_n\to A'_1\to 0.\]
		Note that the flatness of $N/f^nN$ over $A/f^nA$ implies the flatness of  $N'_n$ over $A'_n$. Tensoring with $N'_n$, we get  
		\[0\to N'_{n-1}\xra{\times f} N'_n\to N'_1\to 0.\]
		Now let $x\in N'$ with $fx=0$. We get then  $x\in f^{n-1}N'$ for any integer $n\geqslant1 $, and thus $x=0$ by the completeness of $N'$. This finishes the proof of (3) in the case of reduced prismatic type.

		Assume now that $(A,J)$ is of prismatic type.  By Remark~\ref{R:prismatic-reduction} ,  we may assume that $J=(\xi, f)$ such that $\xi, f\in A$ are nonzero divisor and $A/\xi A$ has bounded $f^{\infty}$-torsion.
	Let $c\geqslant 0$ be an integer such that $(A/\xi A)[f^{\infty}]=(A/\xi A)[f^c]$. For any integer $n\geqslant 1$, we put 
		$A_n=A/(f^n,\xi^n)$ and $N_n=N\otimes_AA_n$. 
		Let $\bar x\in A_n[\xi]$ with a lift $ x\in A$. Then one has $\xi x=\xi^ny+f^n z$ for some $y,z \in A$. By our choice of $c$, there exists $z_1\in A$ such that $f^c z=\xi z_1$. If $n\geqslant c$, one gets thus $\xi x=\xi^ny +f^{n-c}\xi z_1$ and hence $x=\xi^{n-1}y+f^{n-c}z_1$ since $\xi$ is a nonzero divisor.  It follows that  $A_n[\xi]\subset \xi^{n-1}A_n+ f^{n-c}A_n$ for all $n\geqslant  c$.
		Since $N_n$ is flat over $A_n$ by assumption, one has \[N_n[\xi]=A_n[\xi]\otimes_{A_n}N_n\subset  ( \xi^{n-1}A_n+ f^{n-c}A_n)\otimes_{A_n}N_n= \xi^{n-1}N_n+f^{n-c}N_n\]
		for all $n\geqslant c$. 
		Now consider the exact sequence
		\[0\to N_n[\xi]\to N_n\xra{\times \xi}N_n\to N/(\xi, f^n)N\to 0\]
		for all $n\geqslant1$.
		Note that the canonical map $N_{n+c+1}[\xi]\to N_n[\xi]$ is zero, hence the inverse system $(N_n[\xi])_n$ is essentially null. Taking inverse limits, one gets a short exact sequence 
		\[0\to N\xra{\times \xi}N\to \varprojlim_nN/(\xi, f^n)N\to 0.\]
		It follows that $N$ is $\xi$-torsion free and $N/\xi N $ is $f$-adically complete. 
		Note that  
		$N\otimes_A^L A/\xi A=N/\xi N$, and hence   
		\[N\otimes_A^L A/J=(N\otimes^L_A A/\xi A)\otimes^L_{A/\xi A} A/J=(N/\xi N)\otimes_{A/\xi A}^L A/J.\]  
		The statement (3) then follows immediately the case of reduced prismatic type applied to $N/\xi N$.

	\end{proof}

	Let $\Mod^{\wedge}_{J}(A)$ denote the subcategory of $\Mod(A)$ consisting of derived $J$-complete $A$-modules. Let   $M, N$ be objects of $\Mod^{\wedge}_J(A)$. Then  $\Hom_A(M,N)$ is also derived $J$-complete  by \cite[\href{https://stacks.math.columbia.edu/tag/0A6E}{Tag 0A6E}]{stacks-project}.  We put 
	\begin{equation}\label{E:completed tensor}
		M\widehat{\otimes}_AN:=H^0((M\otimes_A^L N)^{\wedge}).\end{equation}
	Note that  the functor $M\mapsto M\widehat{\otimes}_AN$ is right exact  (cf.  \cite[\href{https://stacks.math.columbia.edu/tag/0AAJ}{Tag 0AAJ}]{stacks-project}) and   we have  (cf. \cite[Appendix]{chatz})
	\[\Hom_A(L\widehat{\otimes}_AM, N)=\Hom_A(L, \Hom_A(M,N))\] 
	for all objects $L,M,N\in \Mod^{\wedge}_J(A)$.
	There is a canonical surjection  from  $M \widehat{\otimes}_AN$ to the classical  $J$-adic completion of $M\otimes_AN$, which is in general not an isomorphism.  Let $\FMod^{\wedge}_J(A)$ be the full subcategory of $\Mod^{\wedge}_J(A)$ consisting of $J$-completely flat $A$-modules.

	%However, we have the following:

	Let $\phi: A\to B$ be a  derived $J$-complete $A$-algebra. We denote by \begin{equation}\label{E:base change functor}\phi^*: \Mod_J^{\wedge}(A)\to \Mod_{JB}^{\wedge}(B)
	\end{equation} 
	the base change functor $M\mapsto M\widehat{\otimes}_{A}B$. Then $\phi^*$ is the left adjoint of the functor of  restriction of scalars $\Mod_{JB}^{\wedge}(B)\to \Mod_J^{\wedge}(A)$.

	\begin{lemma}
		Let $\phi: A\to B$ and $\psi: B\to C$ be morphisms of derived $J$-complete $A$-algebras. Then for any $M\in \Mod^{\wedge}_{J}(A)$, there is  a canonical isomorphism 
		\[\psi^*\phi^*(M)\simeq (\psi\circ \phi)^*(M).\]
	\end{lemma}
	
	\begin{proof}
		Consider the distinguished triangle \[\tau_{\leqslant -1}(M{\otimes}^{L}_{A}B)^{\wedge}\to (M{\otimes}^{L}_{A}B)^{\wedge}\to \phi^*(M)=H^0((M{\otimes}^{L}_{A}B)^{\wedge})\to .\]
		Taking the derived tensor product with $C$ and  derived $J$-completion, one gets 
		\[
		\big(\tau_{\leqslant -1}(M{\otimes}^{L}_{A}B)^{\wedge}\otimes^L_BC\big)^{\wedge}\to \big((M{\otimes}^{L}_{A}B)^{\wedge}\otimes^L_BC\big )^{\wedge}\to (\phi^*(M)\otimes_B^LC)^{\wedge}\to.
		\]
		If we write $J=(f_1,\dots, f_r)$, then one has  
		\begin{align*}
			\big((M{\otimes}^{L}_{A}B)^{\wedge}\otimes^L_BC\big )^{\wedge}&=R\lim_n \big((M\otimes_A^LB)^{\wedge}\otimes^L_B\Kos(C; f_1^n,\cdots, f_r^n)\big)\\
			&\simeq R\lim_n\big((M{\otimes}^{L}_{A}\Kos(B; f_1^n, \dots, f_r^n)\otimes_B^LC\big)\\
			&\simeq R\lim_n\big(M{\otimes}^{L}_{A}\Kos(C; f_1^n,\dots, f_r^n)\big)\\
			&=(M\otimes_A^LC)^{\wedge},
		\end{align*}
		where the second isomorphism uses Lemma~\ref{L:completion}. One deduces then
		\[H^0\big(\tau_{\leqslant -1}(M{\otimes}^{L}_{A}B)^{\wedge}\otimes^L_BC\big)^{\wedge}\to  (\psi\circ \phi)^*(M)\to \psi^*\phi^*(M)\to H^1\big(\tau_{\leqslant -1}(M{\otimes}^{L}_{A}B)^{\wedge}\otimes^L_BC\big)^{\wedge}\]
		Since  $\big(\tau_{\leqslant -1}(M{\otimes}^{L}_{A}B)^{\wedge}\otimes^L_BC\big)^{\wedge}$ is concentrated in degree $\leqslant -1$ by  \cite[\href{https://stacks.math.columbia.edu/tag/0AAJ}{Tag 0AAJ}]{stacks-project}, the Lemma follows immediately. 
	\end{proof}
	%Note that the functor $\phi^*$ does not necessarily send $\FMod_{J}^{\wedge}(A)$ to $\FMod_{JB}^{\wedge}(B)$.
	%Fortunately, for pairs of (reduced) prismatic type, this pathology does not occur.

	\begin{lemma}\label{L:stability flatness}
		Let  $(A,J)$ be a pair  of prismatic type (resp. of reduced prismatic type), and $\phi: A\to B$ be an $A$-algebra such that $(B, JB)$ is also of  prismatic type (resp. of  reduced prismatic type).
		
		\begin{enumerate}
			\item  If   $M,N\in \FMod^{\wedge}_J(A)$,  then $(M{\otimes}^L_AN)^{\wedge}\cong M\widehat{\otimes}_AN$ is $J$-completely flat and it coincides  with the classical $J$-adic completion of $M\otimes_AN$. 
			
			\item The base change functor $\phi^*$ sends  $\FMod^{\wedge}_{J}(A)$ to  $\FMod_{JB}^{\wedge}(B)$.

		\end{enumerate}
		
	\end{lemma}
	
	\begin{proof}
		(1)	By Proposition~\ref{P:completely-flat}, we have  
		\begin{align*}
			(M{\otimes}^L_AN)^{\wedge}=R\lim_n (M\otimes_A^LN\otimes^L_AA/J^n)= \lim_n(M\otimes_A N)/J^n(M\otimes_A N),
		\end{align*}
		where the last equality uses the $J$-complete flatness of $M\otimes^L_AN$. The Lemma follows immediately.
		
		(2) Let $M$ be an object of $\FMod^{\wedge}_{J}(A)$. Then $(M\otimes_A^LB)^{\wedge}$ is $JB$-completely flat. By Proposition~\ref{P:completely-flat}(2), 
		we have $M\widehat{\otimes}_AB=(M\otimes_A^LB)^{\wedge}$.
	\end{proof}

	%\end{proof}
	
	Let $B^{\otimes \bullet}$ be the \v{C}ech nerve of $\phi: A\to B$  in the category of $J$-completely flat $A$-algebras, i.e. $B^{\otimes \bullet}$ is the cosimplicial object of $J$-completely flat $A$-algebras with  its $n$-component for $n\geqslant 0$ given by 
	\[B^{\otimes n}:=\underbrace{B\widehat{\otimes}_A\cdots \widehat{\otimes}_AB}_{\text{$(n+1)$-fold}}.\]
	%and for each non-descreasing map $\varphi: \{0,1, \dots, n\}\to \{0,1, \dots, m\}$, the corresponding map   $\varphi_{B^{\otimes \bullet}}: B^{\otimes n}\to B^{\otimes m}$ sends an element at $i$-th copy of $B$ to the same element at $\varphi(i)$-th copy of $B$ (cf.  \cite[\href{https://stacks.math.columbia.edu/tag/016N}{Tag 016N}]{stacks-project}).
	In particular, for any integer $i$ with $0\leqslant i\leqslant n$, one has a map of $A$-algebras   $\delta^{n}_i: B^{\otimes (n-1)}\to B^{\otimes n}$
	given by 
	\[\delta^n_{i}\colon
	\quad b_0\otimes \cdots \otimes b_{n-1}\mapsto b_0\otimes \cdots\otimes  b_{i-1}\otimes 1\otimes b_{i}\otimes \cdots \otimes b_{n-1}.\]
	We have thus a diagram of morphisms of derived $J$-complete $A$-algebras: 
	\[\begin{tikzcd}
		B \arrow[r, shift left, "\delta^1_0"]
		\arrow[r, shift right, "\delta^1_1"'] & B\widehat{\otimes}_AB 
		\arrow[r, shift left=2, "\delta^2_0"]
		\arrow[r,"\delta^2_1" description]
		\arrow[r, shift right=2, "\delta^2_2"']
		& B\widehat{\otimes}_AB\widehat{\otimes}_{A}B,
	\end{tikzcd}
	\]
	Let $\mu:B\widehat{\otimes}_AB\to B$ denote the canonical diagonal surjection.

	%The following definition is due to \cite[Appendix]{chatz}.
	\begin{definition}\label{D:descent-data}
		
		A \emph{descent pair relative to $\phi : A\to B$ }  consists of 
		
		\begin{itemize}
			\item a derived $JB$-complete $B$-module $M$, 
			\item and an  isomorphism of $B\widehat{\otimes}_{A}B$-modules 
			\[\varepsilon: \delta^{1,*}_0(M)=M\widehat{\otimes}_AB\xra{\sim} \delta^{1,*}_1(M)=B\widehat{\otimes}_AM, \]
			called a descent datum on $M$,  such that 
			\[\mu^*(\varepsilon)=\id_M,\quad \delta_{1}^{2,*}(\varepsilon)=\delta_{2}^{2,*}(\varepsilon)\circ \delta_{0}^{2,*}(\varepsilon).\]
		\end{itemize}
		A morphism of   descent pairs $f:(M_1,\varepsilon_1)\to (M_2, \varepsilon_2)$ is a morphism of $B$-modules $f: M_1\to M_2$ such that $\varepsilon_2\circ \delta^{1,*}_1(f)=\delta_0^{1,*}(f)\circ \varepsilon_1$.
		We denote by  $\Desc_{B/A}$  the category of  descent pairs relative to $\phi$.
		
	\end{definition}
	For an object $N\in \Mod_J^{\wedge}(A)$, there is a canonical isomorphism  
	\[\varepsilon_N: \delta_0^{1,*}(N\widehat{\otimes}_AB)\simeq N\widehat{\otimes}_A(B\widehat{\otimes}_AB)\simeq \delta_1^{1,*}(N\widehat{\otimes}_AB)\]
	which clearly satisfies the axiom for a descent pair.
	We get thus  a functor 
	\[\phi^{\natural}: \Mod^{\wedge}_J(A)\to \Desc_{B/A}, \quad  N\mapsto (N\widehat{\otimes}_AB,\varepsilon_N).\]
	A descent pair relative to $\phi$ is called effective if it lies in the essential image of $\phi^{\natural}$.

	%A descent datum $(M,\theta)$ is called $I$-flat if $M$ is $I$-completely flat.
	
	%\item
	% A morphism of two descent data $f: (M_1, \theta_1)\to (M_2, \theta)$ is a morphism of $B$-modules $f: M_1\to M_2$ such that 

	%\item 
	%\end{enumerate}
	
	%We denote by $\Desc_{B/A}$ the category of descent data of derived $I$-complete modules relative to $B/A$. 

	\begin{proposition}\label{P:flat-descent}
		Let $(A,J)$ be a pair of (reduced) prismatic type, 
		and $\phi: A\to B$ be a $J$-completely faithfully flat derived $J$-complete $A$-algebra. 
		
		\begin{enumerate}
			\item Let   $M\in \FMod_{J}^{\wedge}(A)$, and  $\underline{s}(M\widehat\otimes_AB^{\otimes \bullet})$ denote the complex associated to the cosimplicial object $M\widehat\otimes_AB^{{\otimes} \bullet}$. Then the canonical map \[M\xra{\sim} \underline{s}(M\widehat\otimes_AB^{\otimes \bullet})\]
			is a quasi-isomorphism. In particular, one has a equalizer diagram of $A$-modules
			\[\begin{tikzcd}
				M\ar[r] &M\widehat{\otimes}_AB \arrow[r, shift left]
				\arrow[r, shift right] & M\widehat{\otimes}_AB\widehat{\otimes}_AB.
			\end{tikzcd}\]

			\item Let $\FDesc_{B/A}^{\wedge}$ be the subcategory of $\Desc_{B/A}$ consisting of objects $(M,\theta)$ with $M$ an $JB$-completely flat $B$-module. Then the functor $\phi^{\natural}$ induces an equivalence of categories 
			\[
			\FMod_{J}^{\wedge}(A)\xra{\sim } \FDesc_{B/A}^{\wedge}.
			\]
		\end{enumerate}

	\end{proposition}

	\begin{proof}
		As usual, we treat here only the case of prismatic type, and that of reduced prismatic is similar and much simpler.  
		
		For (1), it suffices to prove the statement for $M=A$, the general case being obtained by applying $M\otimes^L_A\_$ and taking derived $J$-completion. By Lemma~\ref{L:flat-local}(2), it suffices to prove that $B\xra{\sim }\big(\underline{s}(B^{\otimes \bullet})\otimes^L_AB\big)^{\wedge}$ is a quasi-isomorphism. But note that   \[\big(\underline{s}(B^{\otimes \bullet})\otimes^L_AB\big)^{\wedge}\simeq \underline{s}(B^{\otimes \bullet}\widehat \otimes_AB), \]
		where $B^{\otimes \bullet}\widehat \otimes_AB$ is nothing but  the \v{C}ech nerve of $\delta_{1}^1: B\to B\widehat\otimes_AB$, which admits a right inverse  given by the multiplication map $B\widehat\otimes_A B\to B$. It follows from  \cite[\href{https://stacks.math.columbia.edu/tag/019Z}{Tag 019Z}]{stacks-project} that  $B^{\otimes \bullet}\widehat \otimes_AB$ is homotopy equivalent to the constant cosimplicial object $B$. Hence the canonical map $B\xra{\sim}\underline{s}(B^{\otimes \bullet}\widehat \otimes_AB)$ is also a homotopy equivalence of complexes.
		
		%The Proposition follows Note that for any integer $n\geqslant1 $, $A/I^n\to B/I^nB$ is faithfully flat by Proposition~\ref{P:completely-flat}(2). 
		For (2), the full faithfulness of the restriction of $\phi^{\natural}$ to $\FMod_{J}^{\wedge}(A)$ follows easily from (1). It remains to prove the essential surjectivity. Let $(M, \varepsilon)$ be an object of $\FDesc_{B/A}^{\wedge}$. For any integer $n\geqslant 1$, let $M_n:=M/J^nB$ and 
		\[\varepsilon_n: \delta_0^{1,*}(M_n)=M_n\otimes_{A/J^n}B/J^nB\xra{\sim} \delta_1^{1,*}(M_n)=B/J^nB\otimes_{A/J^n} M_n \]
		be the reduction modulo $J^n(B\widehat{\otimes}_AB)$ of $\varepsilon$. Then $(M_n, \varepsilon_n)$ is a classical descent datum relative to  $A/J^n\to B/J^nB$. By the classical fpqc-descent, the descent datum $(M_n, \varepsilon_n)$  comes from a flat  $A/J^n$-module $N_n$. It is clear that $N_{n+1}\otimes_{A/J^{n+1}}A/J^n\simeq  N_n$. We put $N=\varprojlim_n N_n$. By \cite[\href{https://stacks.math.columbia.edu/tag/09B8}{Tag 09B8}]{stacks-project}, $N$ is $J$-adically complete and  $N_n=N/J^nN$ for all $n\geqslant 1$. 
	%	Indeed, let $\rF^n N$ denote the kernel of the canonical surjection $N\to N_n$. We need to show that $\rF^nN=J^nN.$ It is clear that $J^nN\subset\rF^n N$.
		% To see the inverse inclusion, we choose a set of generators $a_1,\cdots, a_r$ of the ideal $J^n$, and consider the surjective map 
	%	\[
	%	N^{\oplus r}\to J^nN: (x_i)_i\mapsto \sum_{i=1}^ra_ix_i. 
	%	\]
	%	It suffices to see that the composite $f: N^{\oplus r}\to J^nN\to \rF^nN$ is surjective.  Note that $f$  sends $(\rF^kN)^{\oplus r}$ to $\rF^{n+k}N$ for all $k\geqslant1$, and  $N^{\oplus r}$  and $\rF^nN $ are both separated and complete for the topology defined by the corresponding filtration. By \cite[Chap. III, \S 2, no. 8, Corollary 2 to Theorem 1]{Bourbaki}, it suffices to show that induced map on the graded pieces 
	%	\[
	%	\gr^{k}_{\rF}f\colon (\gr^{k}_{\rF}N)^{\oplus r}\to \gr^{n+k}_{\rF} N
	%	\]
	%	is surjective for any $k\geqslant 1$.  	By an easy diagram chasing, one finds a canonical isomorphism 
	%	$$\gr_{\rF}^k N=\rF^kN/\rF^{k+1}N\xra{\sim} \ker (N_{k+1}\to N_k)= J^kN_{k+1}.$$
	%	Using the fact that $N_{k+1}=N_{n+k+1}/J^{k+1}N_{n+k+1}$,  the surjectivity of $\gr^{k}_{\rF}f$ follows immediately. 
		By Proposition~\ref{P:completely-flat}(3), $N$ is $J$-completely flat over $A$, and hence $(B{\otimes}_A^LN)^{\wedge}$ is $JB$-completely flat. By Proposition~\ref{P:completely-flat}(2), we have 
		\[(N\widehat{\otimes}_AB)\otimes_BB/J^nB=(N{\otimes}_A^LB)^{\wedge}\otimes_B^{L}B/J^nB\simeq N_n\otimes_{A/J^n}B/J^nB\simeq M_n.\]
		Passing to the limit,  we get thus an isomorphism of descent data  $(N\widehat{\otimes}_AB, \varepsilon_N)\simeq (M,\varepsilon)$. %This proves the Proposition.
		
	\end{proof}

	\begin{remark}
		As pointed by the anonymous referee, if $(A,J)$ is a pair of (reduced) prismatic type,   Propsosition~\ref{P:completely-flat} and \cite[\href{https://stacks.math.columbia.edu/tag/09B8}{Tag 09B8}]{stacks-project} actually imply that $N\mapsto (N\otimes_AA/J^n)_{n\geqslant1}$  establishes an equivalence of categories  between $\FMod^{\wedge}_J(A)$ and  the category of inverse systems  of $A$-modules $(N_{n})_{n\geqslant 1}$ such that  each $N_n$  is flat over $A/J^n$ and $N_{n+1}\otimes_{A/J^{n+1}}A/J^n\simeq N_n$. 
	\end{remark}

	\section{Prismatic Crystals and Finiteness Theorems}
	
	%In this section, we recall some preliminaries on the $q$-crystalline site and prismatic site introduced in \cite{BS}.
	
	In this section, we fix a prime number $p$. All the rings are supposed to be $\ZZ_{(p)}$-algebras.
	
	\subsection{Prismatic site}
	%Recall that a ring $R$ is quasi-syntomic if $R$ is $p$-complete with bounded $p^{\infty}$-torsion and if the cotangent complex $L_{R/\ZZ_p}$ has complete Tor-amplitude in $[-1, 0]$. A $p$-adic formal scheme $X$ is called quasi-syntomic if it is covered by affine formal schemes $\Spf(R)$ with $R$ quasi-syntomic. 

	Let $(A,I)$ be a bounded prism in the sense of  \cite[Def. 3.2]{BS}, and  $X$ be a  $p$-adic formal scheme over $\Spf(A/I)$. 
	We recall first the prismatic site of $X$ relative to $A$, denoted by $(X/A)_{\prism}$, introduced in \cite[Def. 4.1]{BS}:
	
	\begin{itemize}
		\item The underlying category of $(X/A)_{\prism}$ is the opposite  of bounded prisms $(B,IB)$ over $(A,I)$  together with a map of $p$-adic formal schemes  $\Spf(B/IB)\to X$ over $A/I$; the notion of morphism in $(X/A)_{\prism}$ is the obvious one. We shall often denote such an object by 
		\[
		(\Spf(B)\leftarrow \Spf(B/IB)\ra X);
		\]
		or by $(B\to  B/IB \leftarrow R)$ if $X=\Spf(R)$ is affine. 
		%If there is no confusion on the structure map $\Spf(B/J)\to X$,   we will simply write  such  an object as $(B,J)$. 

		\item A map 
		\[
		(\Spf(C)\leftarrow \Spf(C/IC)\ra X)\to (\Spf(B)\leftarrow \Spf(B/IB)\ra X)
		\]
		in $(X/A)_{\prism}$ is a flat cover if the underlying map  of $\delta$-$A$-algebras $B\to C$ is $(p,I)$-completely
		faithfully flat. 
	\end{itemize}   
	
	Following Grothendieck, we denote by $(X/A)_{\prism}^{\sim}$ the associated topos. From  Proposition~\ref{P:flat-descent},  it follows that there is a canonical embedding $(X/A)_{\prism}\to (X/A)_{\prism}^{\sim}$.
	
	\begin{remark}\label{R:fiber-product}
		Let $f: V:=(\Spf(C)\leftarrow \Spf(C/IC)\ra X)\to U:=(\Spf(B)\leftarrow \Spf(B/IB)\ra X)$ and $g:W:= (\Spf(D)\leftarrow \Spf(D/ID)\ra X)\to U$ be two morphisms in $(X/A)_{\prism}$. Assume that one of $f$ and $g$, say $g$,  is $(p,I)$-completely flat (i.e. the underlying map of $\delta$-$A$-algebras $B\to D$ is $(p,I)$-completely flat). Then $E:=(C\otimes^L_BD)^{\wedge}$ concentrated in degree $0$ and it is a $(p,I)$-completely flat $\delta$-$C$-algebra (cf. Lemma~\ref{L:stability flatness}); moreover,  there exists a canonical map 
		\[\Spf(E/IE)=\Spf(C/IC)\times_{\Spf(B/IB)}\Spf(D/ID)\to \Spf(B/IB)\to X.\] Then the object  $(\Spf(E)\leftarrow \Spf(E/IE)\to X)$ gives the fibre product $V\times_UW$. 
	\end{remark}

	We denote by $\cO_{(X/A)_{\prism}}$ (resp. $\overline{\cO}_{\prism}=\overline\cO_{(X/A)_{\prism}}$) the structural sheaf (resp. the reduced structural sheaf) on $(X/A)_{\prism}$ which sends an  object $(B,J)$ in $(X/A)_{\prism}$ to $B$ (resp. to $B/J$). If no confusion arises, we will simply write $\cO_{{\prism}}$ for $\cO_{(X/A)_{\prism}}$ and $\overline \cO_{\prism}$
	for $\overline{\cO}_{(X/A)_{\prism}}$.

	\begin{definition}\label{D:crystals}
		An \emph{ $\cO_{\prism}$-crystal} (resp. an $\overline{\cO}_{\prism}$-crystal) on $(X/A)_{\prism}$ is a sheaf of $\cO_{\prism}$-modules $\calF$ on $(X/A)_{\prism}$ such that
		\begin{itemize}
			\item for each object $(\Spf(B)\leftarrow \Spf(B/IB)\to X)$ of $(X/A)_{\prism}$,  the evaluation \[\calF_B:=\calF(\Spf(B)\leftarrow \Spf(B/IB)\to X)\] is a derived $(p,I)$-complete  and \emph{$(p, I)$-completely flat} $B$-module (resp. a derived $p$-complete and $p$-completely flat $B/I$-module), 
			
			\item  for any morphism 
			$$ (\Spf(C)\leftarrow \Spf(C/IC)\to X)\xra{f} (\Spf(B)\leftarrow \Spf(B/IB)\ra X),$$
			in $(X/A)_{\prism}$ the canonical linearized transition map 
			\[
			c_f(\calF):	f^*(\calF_B):=\calF_B\widehat\otimes_{B}C\ra \calF_C\]
			(resp. $c_{\bar f}(\calF): {\bar f}^*(\calF_B):=\calF_{B}\widehat{\otimes}_{B/IB}C/IC\to \calF_C$)
			is an isomorphism, where $\widehat{\otimes}$ is the completed tensor product for  the ideal $(p,I)$ (resp. for the ideal $(p)$) defined in \eqref{E:completed tensor}.
			
			An $\cO_{\prism}$-crystal  (resp.  an $\overline{\cO}_{\prism}$-crystal) $\calF$ is  called \emph{locally free of finite rank} if $\calF_{B}$ is a locally free $B$-module (resp. $B/IB$-module) of finite rank for each object $(\Spf(B)\leftarrow \Spf(B/IB)\to X)$. 
		\end{itemize}
		
		%	A  prismatic crystal $\calF$ is called locally free of finite type if so is each $\calF_B$  as a $B$-module. 
		We denote %by $\Qcoh(X/A)_{\prism}$ the category of quasi-coherent crystals on $(X/A)_{\prism}$, and 
		by $\CR((X/A)_\prism, \cO_{\prism})$ (resp. $\CR((X/A)_{\prism},\overline\cO_{\prism})$) the category of  $\cO_{\prism}$-crystals (resp. $\overline\cO_{\prism}$-crystals). 
		
		%(2) For an object $\calF$ in $\CR(X/A)_{\prism}$, its Frobenius twist, denoted by $\phi^*\calF$, is  the sheaf that sends an object $(B,J)$ in $(X/A)_{\prism}$ to $\calF_B\widehat \otimes_{B,\phi_B}B$. It is clear that $\phi^*\calF$ is also an object in $\CR(X/A)_{\prism}$.
		
		%(3) A prismatic $F$-crystal is an object $\calF$ in $\CR(X/A)_{\prism}$ equipped with an isomorphism 
		%\[(\phi^*\calF)[\frac{1}{I}]\xra{\sim} \calF[\frac{1}{I}].\]
		%We denote by $F-\CR(X/A)_{\prism}$ the category of  prismatic $F$-crystals.
	\end{definition}

	\begin{lemma}\label{L:crystal-data}
		The functor $\calF\mapsto (\{\calF_B\}, \{c_f(\calF)\})$ induces an equivalence of $\CR((X/A)_{\prism}, \cO_{\prism})$ and the category of the data $(\{M_B\},  \{c_f\})$, where 
		\begin{itemize}
			\item $\{M_B\}$ is the collection of  derived  $(p,I)$-complete and $(p, I)$-completely flat $B$-modules  $M_B$ corresponding to each object $(\Spf(B)\leftarrow \Spf(B/IB)\to X )$ of $(X/A)_{\prism}$;
			
			\item $\{c_f\}$ is the collection of isomorphisms of $C$-modules
			\[c_f: C\widehat{\otimes}_BM_B\xra{\sim} M_C\]
			for each morphism $f: (\Spf(C)\leftarrow \Spf(C/IC)\to X)\to (\Spf(B)\leftarrow \Spf(B/IB)\ra X) $ in $(X/A)_{\prism}$ such that 
			\begin{itemize}
				\item $c_{f}$  is the identity map if $f$ is the identity of an object of $(X/A)_{\prism}$;
				\item   the cocycle condition 
				\[c_{f\circ g}=c_g\circ g^*(c_f)\]
				is satisfied	for a composition of morphisms.
			\end{itemize}
		\end{itemize}
		
	\end{lemma}
	
	\begin{proof}
		It is sufficient to construct a quasi-inverse to the evaluation  functor $\calF\mapsto (\{\calF_B\}, \{c_f(\calF)\})$. Given a datum $(\{M_B\}, \{c_f\})$ as above,   the cocycle condition guarantees that 
		\[\calM:(\Spf(B)\leftarrow \Spf(B/IB)\to X )\mapsto  M_B\]
		together with transition map 
		\[\calM_f: M_B\to f^*(M_B)=M_B\widehat{\otimes}_BC\xra{c_f} M_C\]
		for each morphism $f: (\Spf(C)\leftarrow \Spf(C/IC)\to X)\to (\Spf(B)\leftarrow \Spf(B/IB)\ra X) $ in $(X/A)_{\prism}$ indeed 
		defines a presheaf on $(X/A)_{\prism}$. 
		Then Proposition~\ref{P:flat-descent}(1) implies that $\calM$ is indeed a sheaf. The fact that $\calM$ is  an $\cO_{{\prism}}$-crystal follows immediately from the conditions on $\{c_f\}$.
	\end{proof}
	
	%\begin{remark}
	%	In this article, we restrict ourselves to $(p, I)$-completely flat objects in order to apply the descent results Prop.  \ref{P:flat-descent}. %One can of course consider
	%\end{remark}
	
	\begin{remark}
		There is an obvious analogue of this Lemma for $\overline\cO_{\prism}$-crystals by requiring   $c_{\bar f}: C/IC\widehat{\otimes}_{B/IB}M_B\to M_{C}$ to be an isomorphism for each morphism $f$ as above.
	\end{remark}
	
	\begin{example}
		Assume that $(A,I)$ is a perfect prism, and   $X=\Spf(R)$ with $R$ quasi-regular semi-perfectoid  and $A/I\to R$ is surjective.  Then the prismatic cohomology $\prism_{R/A}:=R\Gamma((X/A)_{\prism}, \cO_{\prism})$ is concentrated in degree $0$ and $(\prism_{R/A},I\prism_{R/A} )$ becomes the final object in $(X/A)_{\prism}$ (see \cite[\S 3.4]{AB}). Therefore,  $\CR((X/A)_{\prism},\cO_{\prism})$ is equivalent to the category of derived $I$-complete and completely $I$-flat $\prism_{R/A}$-modules.
	\end{example}
	
	\subsection{Functoriality}\label{S:functoriality} Let $(A,I)\to (A',I')$ be a morphism of bounded prisms. Let 
	\[\xymatrix{Y\ar[r]^f\ar[d]&X\ar[d]\\
		\Spf(A'/I')\ar[r] &\Spf(A/I)
	}\]
	be a commutative diagram of $p$-adic formal schemes. Then $f$ induces a morphism of sites 
	\[f_{\sharp}: (Y/A')_{\prism}\to (X/A)_{\prism}\]
	given by  \[(\Spf(B')\leftarrow \Spf(B'/J')\to Y)\mapsto (\Spf(B')\leftarrow \Spf(B'/J')\to Y\xra{f} X).\]  
	It is easy to see that $f_{\sharp}$ is cocontinuous,  hence it induces a morphism of topoi:
	\[f_{\prism}=(f^{-1}_{\prism},f_{\prism,*}): (Y/A')_{\prism}^{\sim}\to (X/A)_{\prism}^{\sim}.\]	
	For a sheaf $\calF$ on $(X/A)_{\prism}$, its inverse image is given by  
	\[(f^{-1}_{\prism}\calF)(B',J')=\calF(f_{\sharp}(B',J'))\]
	For a sheaf $\calE$ on $(Y/A')_{\prism}$, its direct image $f_{\prism,*}(\calE)$ is described as follows. For an object $(B,J)$ in $(X/A)_{\prism}$, let $Y_B=Y\times_{X}\Spf(B/J)$ and $B'=B\widehat{\otimes}_AA'$. Assume that  $(B',IB')$  is  a bounded prism (which is the case if either $A'$ or $B$ is $(p,I)$-completely flat over $A$.) %({\color{blue}{Actually, I don't know how to show this. If this is not true, one needs to  modify the definition of $B'$}}) and one has a canonical map $Y_B\to \Spf(B'/IB')$. 
	Then the direct image of a sheaf $\calE$ on $(Y/A')_{\prism}$ is given by 
	\begin{equation}\label{E:formula-relative}
		f_{\prism,*}(\calE)(B,J)=\Gamma((Y_B/B')_{\prism}, \calE|_{(Y_B/B')_{\prism}})
		\end{equation}
	where $\calE|_{(Y_B/B')_{\prism}}$ is the pullback of $\calE$ to $(Y_B/B')_{\prism}$.
	
	For an $\cO_{(X/A)_\prism}$-module $\calF$, we put 
	\[f^*_{\prism}\calF:=f^{-1}_{\prism}\calF\widehat\otimes_{f^{-1}_{\prism}\cO_{(X/A)_{\prism}}}\cO_{(Y/A')_{\prism}}.\]
	Then the functor $f^*_{\prism}$ induces the pullback map for prismatic crystals:
	\[f^*_{\prism}: \CR((X/A)_{\prism}, \cO_{(X/A)_\prism})\to \CR((Y/A')_{\prism},\cO_{(Y/A')_\prism}).\]
	Similarly, the formula
	\[\bar f^*_{\prism}: \overline \calF\mapsto f^{-1}\overline\calF\otimes_{f^{-1}_{\prism}\overline{\cO}_{(X/A)_{\prism}}}\overline{\cO}_{(Y/A')_{\prism}}\]
	defines a pullback functor 
	\[\bar f_{\prism}^*\colon \CR((X/A)_{\prism}, \overline{\cO}_{(X/A)_{\prism}})\to \CR((Y/A')_{\prism}, \overline{\cO}_{(Y/A')_{\prism}}).\]
	
	\bigskip We can now state the main results of this article.
	Let $X_{\Et}$ be the big \'etale site of $X$ consisting of all $p$-adic formal schemes over $X$. Let  $\nu_{X/A}: (X/A)_{\prism}^{\sim}\to X_{\Et}^{\sim}$ be the canonical projection. For any  sheaf $\calF$ on $(X/A)_{\prism}$ and any object $U$ of $X_{\Et}$, we have  
	\[(\nu_{X/A,*}\calF)(U)= \Gamma((U/A)_{\prism}, \calF_U),\]
	where $\calF_U$ is the inverse image of $\calF$ under the natural map of topos $(U/A)^{\sim}_{\prism}\to (X/A)^{\sim}_{\prism}$.
	
	\begin{theorem}\label{T:Obar-crystal}
		Let $X$ be a smooth $p$-adic formal scheme over $A/I$ of relative dimension $n$. Let $\calE$ be an $\overline \cO_{\prism}$-crystal locally free of finite rank on $(X/A)_{\prism}$. Then the following statements hold:
		
		\begin{enumerate}
			\item  $R\nu_{X/A,*}(\calE)$ is a perfect complex of $\cO_X$-modules with tor-amplitude in $[0,n]$.
			
			\item Let $(A,I)\to (A',I')$ be a morphism of  bounded prisms. Consider the 
			cartesian diagram
			\begin{equation}\label{E:cartesian diagram}
				\xymatrix{X'\ar[d]\ar[r]^f &X\ar[d]\\
					\Spf(A'/I')\ar[r] & \Spf(A/I).
				}
			\end{equation}
			The canonical base change map 
			\[f^{-1}R\nu_{X/A,*}*( \calE)\otimes^{L}_{f^{-1}\cO_{X}}\cO_{X'}\xra{\sim} R\nu_{X'/A',*}(\bar f_{\prism}^* \calE)\]
			is an isomorphism.
		\end{enumerate}
		
	\end{theorem}
	
	The proof of this Theorem will be given in Subsection~\ref{S:proof of main theorem} after some local preparations.
	%We will  give an explicit local description of $R\nu_{X/A,*}(\overline \calF)$ in terms of the de Rham of certain Higgs module. 
	For the moment, one can deduce immediately from Theorem~\ref{T:Obar-crystal} the following finiteness result on the cohomology of an $\cO_{{\prism}}$-crystal.  
	
	\begin{theorem}\label{T:finiteness-crystal}
		Let  $X$ be a  proper and smooth $p$-adic formal scheme  over $\Spf(A/I)$ of relative dimension $n$.
		Let $\calF$ be an $\cO_{\prism}$-crystal locally free of finite rank  on $(X/A)_{\prism}$. 
		Then   $R\Gamma((X/A)_{\prism}, \calF)$ is a perfect complex of $A$-modules with tor-amplitude in $[0,2n]$.
		Moreover, if   $(A,I)\to (A',I')$ is a morphism of $p$-torsion free bounded prisms that induces the cartesian diagram \eqref{E:cartesian diagram},   then the canonical base change map 
		\[R\Gamma((X/A)_{\prism}, \calF)\otimes^L_AA'\xra{\sim}R\Gamma((X'/A')_{\prism}, f_{\prism}^*\calF)\]
		is an isomorphism.
		
	\end{theorem}

	\begin{proof}
		Applying Theorem~\ref{T:Obar-crystal}(1) to $\overline \calF:=\calF/I\calF$, we see that $R\nu_{X/A,*}(\overline \calF)$ is a perfect complex of $\cO_X$-modules with perfect amplitude in $[0,n]$. Since $X$ is assumed to be proper and smooth,  it follows that 
		\[R\Gamma((X/A)_{\prism},\calF)\otimes_A^LA/I\xra{\sim} R\Gamma((X/A)_{\prism},\overline\calF)\cong R\Gamma(X_{\et}, R\nu_{X/A,*}(\overline \calF))\]
		is a perfect complex of $A/I$-modules with perfect amplitude $[0, 2n]$. We then conclude by  \cite[\href{https://stacks.math.columbia.edu/tag/07LU}{Tag 07LU}]{stacks-project} that $R\Gamma((X/A)_{\prism}, \calF)$ is a perfect complex with tor-amplitude in $[0,2n]$.
		
		For the second part of the Theorem, according to the derived Nakayama Lemma \cite[\href{https://stacks.math.columbia.edu/tag/0G1U}{Tag 0G1U}]{stacks-project}, it suffices to show that 
		\[R\Gamma((X/A)_{\prism}, \calF)\otimes^L_AA'/I'\xra{\sim}R\Gamma((X'/A')_{\prism}, f_{\prism}^*\calF)\otimes^L_{A'}A'/I'\] 
		is an isomorphism.
		But this follows from the following sequence of canonical isomorphisms:
		\begin{align*}
			R\Gamma((X/A)_{\prism},\overline \calF)\otimes^L_{A/I}A'/I'&\cong R\Gamma(X_{\et}, R\nu_{X/A,*}(\overline \calF))\otimes_{A/I}^LA'/I'\\
			&\cong  R\Gamma(X'_{\et}, f^{-1} R\nu_{X/A,*}(\overline \calF)\otimes^L_{f^{-1}\cO_{X}}\cO_{X'})\\
			&\cong R\Gamma(X_{\et}', R\nu_{X'/A',*}(\bar f_{\prism}^*\overline\calF))\\
			&\cong R\Gamma((X'/A')_{\prism}, \bar f^*_{\prism}\calF).
		\end{align*}
		where the second isomorphism is the projection formula for coherent cohomology, and the third one is Theorem~\ref{T:Obar-crystal}(2).
	\end{proof}

It is   standard to  deduce from Theorem~\ref{T:finiteness-crystal} some finiteness results in the relative case. 
\begin{theorem}
	Let   $f:Y\to X$ be a proper and smooth morphism of $p$-adic formal schemes over $\Spf(A/I)$ of relative dimension $n$, and $f_{\prism}: (Y/A)_{\prism}^{\sim}\to (X/A)_{\prism}^{\sim}$ be the associated morphism of topoi. Let $\calF$ be an $\cO_{\prism}$-crystal  locally free of finite rank on $(Y/A)_{\prism}$. Then $Rf_{\prism,*}(\calF)$ is a perfect complex of $\cO_{\prism}$-crystals on $(X/A)_{\prism}$ with tor-amplitude in degree  $[0,2n]$ in the  following sense:
	\begin{itemize}
		\item  for an object $(\Spf(B)\leftarrow \Spf(B/IB)\ra X)$ in $(X/A)_{\prism}$,  its evaluation $Rf_{\prism,*}(\calF)_{B}$ is a perfect complex of $B$-module with tor-amplitude in degree $[0,2n]$;
		\item  for any morphism $\alpha: (\Spf(C)\leftarrow \Spf(C/IC)\ra X)\to (\Spf(B)\leftarrow \Spf(B/IB)\ra X)$ in $(X/A)_{\prism}$, the natural base change map 
		\[
		(Rf_{\prism,*}(\calF)_{B}\otimes^L_BC)^{\wedge}\xra{\sim} Rf_{\prism,*}(\calF)_{C}
		\]
		is an isomorphism. 
	\end{itemize}
	
	\end{theorem}

\begin{proof}
	For an object $(\Spf(B)\leftarrow \Spf(B/IB)\ra X)$, we put $Y_B=Y\times_{X}\Spf(B/IB)$. Then it follows from \eqref{E:formula-relative} that $Rf_{\prism,*}(\calF)_{B}\cong R\Gamma((Y_B/B)_{\prism}, \calF|_{(Y_B/B)_{\prism}})$.  This theorem follows immediately from Theorem~\ref{T:finiteness-crystal}.
	\end{proof}

	\section{Local description of prismatic crystals}\label{S:local-crystal}
	In this section, 
	we fix   a bounded prism $(A,I)$. Let  $X=\Spf(R)$ be an affine smooth $p$-adic formal scheme over $\Spf(A/I)$ of relative dimension $n\geqslant 0$.
	We will make the following assumption: 
	\begin{assumption}\label{A:local-lift}
		The $A/I$-algebra $R$ admits  a lift to a  derived $(p, I)$-complete  $\delta$-$A$-algebra $\widetilde R$ that is formally smooth over $A$. 
	\end{assumption}
	We fix such a lift $\widetilde R$, and let $\widetilde X:=(\widetilde R\to \widetilde R/I\widetilde R\cong R)$ denote the corresponding object in $(X/A)_{\prism}$.

	\begin{lemma}\label{L:final-cover}
		Under Assumption~\ref{A:local-lift}, for any object $(B\to B/IB\leftarrow R)$ in $(X/A)_{\prism}$, the product of $(B\to B/IB\leftarrow R)$ and $\widetilde X$ in  $(X/A)_{\prism}$ exists. Moreover, if we  denote this product by $(\widetilde B\to  \widetilde B/ I\widetilde B\leftarrow R)$, then  
		$\widetilde B$ is  $(p,I)$-completely faithfully flat over $B$. 
		In particular, $\widetilde X$  is a cover of the final object of  $ (X/A)^{\sim}_{\prism}$.
	\end{lemma}
	\begin{proof}
		%	Note that $\tilde R$ is $(p,I)$-completely faithfully flat over $A$. 
		Let $C:=(B\otimes_A^L\widetilde R)^{\wedge}$ be the derived $(p,I)$-completion of $B\otimes_A^L\widetilde R$. Then $C$ is $(p,I)$-completely faithfully flat over $B$. Hence by  Proposition~\ref{P:completely-flat}(2) it is   concentrated in degree $0$ and coincides with the classical $(p,I)$-adic completion of $B\otimes_A\widetilde R$. By \cite[Lemma 2.17]{BS}, there is a unique $\delta$-structure on $C$ compatible with the natural product $\delta$-structure  on $B\otimes_A\widetilde R$. Consider the surjection $C\to B/IB\otimes_{A/I} R\to B/IB$, and denote its kernel by $J$. 
		As $ R$ is  formally smooth over $A/I$ of relative dimension $n$,  $J$ is  locally generated by $I$ and a regular sequence of length $n$ relative to $B$. Applying \cite[Lemma~3.13]{BS}, we get  the prismatic envelope $\widetilde B:=C\{\frac{J}{I}\}^{\wedge}$ which is $(p,I)$-completely faithfully flat over $B$, and commutes with base change in $(B,IB)$. The fact that $(\widetilde B\twoheadrightarrow \widetilde B/I\widetilde B\leftarrow R)$ is the product of $(B\twoheadrightarrow B/IB\leftarrow R)$ and $\widetilde X$ in  $(X/A)_{\prism}$ follows easily from the universal property of the prismatic envelope.

		The second part of the Lemma follows immediately from the following general fact: if $\calC$ is a topos, an object $U\in \calC$  is a cover of the final object of $\calC$ if and only if for any object $V\in \calC$ there exists a cover $W\to V$ such that $W$ admits a morphism to $U$.
	\end{proof}

	\subsection{Simplicial object}\label{S:cosimplicial object}
	For each integer $m\geqslant 0$, let $\widetilde{X}(m)$ be the $(m+1)$-fold self-product of the object $\widetilde X$ in $(X/A)_{\prism}$. By the proof of Lemma~\ref{L:final-cover}, $\widetilde X(m)$ is explicitly given  by as follows.  Let $\widetilde J(m)$ be the kernel of the 
	canonical surjection 
	\[
	\widetilde R^{\otimes (m+1)}:=\underbrace{\widetilde R\widehat{\otimes}_A\dots \widehat{\otimes}_A\widetilde R}_{(m+1) \text{ copies}}\to R.
	\] 
	Since $R$ is formally smooth over $A/I$,  $\widetilde J(m)$ is locally generated by $I$ together with a $(p,I)$-completely regular sequence relative to $A$.
	We let  
	$$\widetilde{R}(m): =\widetilde R^{\otimes (m+1)}\{\frac{\widetilde{J}(m)}{I}\}^{\wedge}$$ 
	be  the prismatic envelope by the construction of \cite[Lemma 3.13]{BS}, and put $R(m):=\widetilde R(m)/I\widetilde R(m)$
	Then we have $\widetilde{X}(m)=(\widetilde R(m)\to \widetilde R(m)/I\widetilde R\leftarrow R )$. %The resulting simiplicial object $\widetilde{X}(\bullet)$ is then the \v{C}ech nerve of $\widetilde{X}$.
	%For an integer $i$ with $0\leqslant i\leqslant n$, let 
	%\[\delta_i^n: \widetilde X(n)\to \widetilde X(n-1)\]
	%be the natural projection map forgetting the $i$-th component.
	% 
	The $(\widetilde R(m), I\widetilde R(m))$'s form naturally a cosimplicial object of bounded prisms $(\widetilde R(\bullet), I\widetilde R(\bullet))$ over $(A,I)$. 
	% By Remark~\ref{R:fiber-product},  the simiplicial object $\widetilde X(\bullet)$ is   nothing but the \v{C}ech nerve of $\widetilde X(1)\to \widetilde X$;  in particular one has  a canonical isomorphism 
	% \begin{equation}\label{E:cosimplicial-iso}
	 %	\widetilde R(m)=\underbrace{\widetilde R(1)\widehat{\otimes}_{\widetilde R} \cdots \widehat{\otimes}_{\widetilde R} \widetilde R(1)}_{m \text{copies}}
%	 \end{equation}
	For an integer $i$ with $0\leqslant i\leqslant m$, let 
	\begin{equation}\label{E:simplicial-projection}
		\delta^m_{i}: \widetilde R(m-1)\to \widetilde R(m)\end{equation}
	denote the map of $\delta$-$A$-algebras corresponding to the strictly increasing map of simplices $[m-1]:=\{0, 1,\dots, m-1\}\to[m]$ that skips  $i$, i.e. $\delta^m_{i}$ is  
	induced by the map $\widetilde R^{{\otimes }m}\to \widetilde R^{{\otimes }(m+1)}$ given by 
	\[b_0\otimes \cdots \otimes b_{m-1}\mapsto b_0\otimes \cdots\otimes  b_{i-1}\otimes 1\otimes b_{i}\otimes \cdots \otimes b_{m-1}.\] %Dually, $\delta_i^n$ corresponds to the projection $\pr_i^n:\widetilde X(m)\to \widetilde X(m-1)$ that ignores the $i$-th copy of $\widetilde X$.
	In particular, we have a diagram of $\delta$-$A$-algebras:
	\[\begin{tikzcd}
		\widetilde R \arrow[r, shift left, "\delta^1_0"]
		\arrow[r, shift right, "\delta^1_1"'] & \widetilde R(1) 
		\arrow[r, shift left=2, "\delta^2_0"]
		\arrow[r,"\delta^2_1" description]
		\arrow[r, shift right=2, "\delta^2_2"']
		& \widetilde R(2).
	\end{tikzcd}
	\]
	Let $\mu: \widetilde R(1)\to \widetilde R$ denote the diagonal surjection  induced by  $\widetilde R\widehat{\otimes}_A \widetilde R\to \widetilde R$ given as $a\otimes b\mapsto ab.$
	
	\begin{remark}\label{R:simplicial product}
		Note that, for each integer $m\geqslant 1$,  one has an isomorphism of objects in $(X/A)_{\prism}$:
		\[
		\widetilde X(m)\cong \underbrace{\widetilde X(1)\times_{\widetilde X}\cdots \times_{\widetilde X}\widetilde X(1)}_{m\text{ copies}},
		\]
		where each $\widetilde X(1)=\widetilde X\times_{(X/A)_{\prism}}\widetilde X$ maps to $\widetilde X$ via the first projection. This implies that 
		\[
		\widetilde R(m)\cong \underbrace{\widetilde R(1)\widehat{\otimes}_{\widetilde R}\widetilde R(1)\widehat{\otimes}_{\widetilde R}\cdots\widehat{\otimes}_{\widetilde R}\widetilde R(1)}_{m \text{ copies}}
		\]
		where each $\widetilde R(1)$ is viewed as an $\widetilde R$-algebra via $\delta^1_1$.
		
	\end{remark}
	%We denote by  
	%$$\bar \delta_i^n: \widetilde R(n-1)/I\widetilde {R}(n-1)\to\widetilde R(n)/I\widetilde {R}(n) $$  the reduction of $\delta^n_i$ modulo $I$.
	
	%Letting $n\geqslant 0$ vary, we get thus a cosimplical object $(\widetilde R(\bullet), I\widetilde{R}(\bullet))$ of prisms over $(A,I)$, or dually a  simplicial object $\widetilde X(\bullet)$ in $(X/A)_{\prism}$, which is nothing but the \v{C}ech nerve of $\widetilde X$. 
	\begin{definition}\label{D:prismatic stratification}
		Let $M$ be a derived $(p,I)$-complete and $(p,I)$-completely flat $\widetilde R$-module. 
		A stratification of $M$  over  $\widetilde{R}(1)$
		is an isomorphism of $\widetilde R(1)$-modules:
		$$\epsilon : \delta_{0}^{1,*}(M)\xra{\sim} \delta_1^{1,*}(M)$$ 
		where 
		$\delta_i^{1,*}$ is the base change functor \eqref{E:base change functor},
		such that the cocyle condition is satisfied:
		\[\mu^*(\epsilon)=\id_M, \quad \delta_{1}^{2,*}(\epsilon)=\delta_{2}^{2,*}(\epsilon)\circ \delta_{0}^{2,*}(\epsilon).\]
		
		We denote by $\Strat(\widetilde R, \widetilde R(1))$  the category of derived $(p,I)$-complete and $(p,I)$-completely flat $\widetilde R$-modules together with a stratification over $\widetilde R(1)$.
	\end{definition}

	\begin{remark}\label{R:rigidification}
		The condition  $\mu^*(\epsilon)=\id_M$ is actually redundant, but we choose to keep it for convenience.  
		 Indeed, by applying the functor  $(\id\otimes \mu)^*$ to the cocycle condition $\delta_{1}^{2,*}(\epsilon)=\delta_{2}^{2,*}(\epsilon)\circ \delta_{0}^{2,*}(\epsilon)$, one gets
		\[\epsilon=\epsilon\circ \delta_{0}^{1,*}(\mu^*(\epsilon))\]
		and hence $\delta_{0}^{1,*}(\mu^*(\epsilon))=\id_{\delta_{0}^{1,*}(M)}$. As $\delta^{1}_{0}: \tilde R\to \tilde R(1)$ is $(p,I)$-completely faithfully flat, one deduces from Proposition~\ref{P:flat-descent} that $\mu^*(\epsilon)=\id_{M}$.
		
	\end{remark}
	
	We have also a variant of Definition~\ref{D:prismatic stratification} and Proposition~\ref{P:crystal-stratification} for $\overline{\cO}_{\prism}$-crystals. In general, if $f$ is morphism of objects over $A$ (e.g. $A$-modules, $A$-algebras or $A$-formal schemes, ...), we denote by $\bar f: B/IB\to C/IC$ its reduction modulo $I$.
	%Indeed, we define an $\overline $
	
	\begin{definition}\label{D:Obar-stratification}
		Let $M$ be a derived $p$-complete and $p$-completely flat $R$-module. 
		A stratification of $M$ over $R(1)=\widetilde R(1)/I\widetilde R(1)$ is an isomorphism of $ R(1)$-modules $$\epsilon:\bar \delta_0^{1,*}(M)\xra{\sim}\bar\delta_1^{1,*}(M)$$ 	
		such that $\mu^*(\epsilon)=\id_M$ and  $\bar \delta^{2,*}_1(\epsilon)=\bar\delta^{2,*}_2(\epsilon)\circ \bar \delta_0^{2,*}(\epsilon)$.

		We denote by  $\Strat( R,  R(1))$  the category of  derived $p$-complete and $p$-completely flat $R$-modules together with a stratification over $ R(1)$.
	\end{definition}

	Let $\calF$ be an $\cO_{\prism}$-crystal  on $(X/A)_{\prism}$, and $\calF(\widetilde X)$ be its value on $\widetilde X$. The crystal condition gives rise to a canonical  isomorphism 
	\[
	\epsilon_{\calF}: \delta_0^{1,*}(\calF(\widetilde X))\xra{c_{\delta_{0}^1}(\calF)} \calF(\widetilde X(1))\xrightarrow{c_{\delta_1^{1}}(\calF)^{-1}}\delta_1^{1,*}\calF(\widetilde{X})
	\]
	which makes $(\calF(\widetilde X), \epsilon_{\calF})$ an 
	object of $\Strat(\widetilde R(1))$. We get thus an evaluation functor 
	\[\ev_{\widetilde X}\colon \CR((X/A)_{\prism},\cO_{\prism})\to \Strat(R, \widetilde R(1))
	\]
	sending $\calF$ to $(\calF(\widetilde X), \epsilon_{\calF})$. 
	Similarly, we have also an evaluation  functor 
	\[\overline{\ev}_{\widetilde X}\colon\CR((X/A)_{\prism},\overline \cO_{\prism})\to \Strat(R,  R(1))\]
	for $\overline\cO_{\prism}$-crystals. 
	
	\begin{proposition}\label{P:crystal-stratification}
		Under the notation above, 
		the functors $\ev_{\widetilde X}$ and $\overline{\ev}_{\widetilde X}$ are both  equivalences of categories.
	\end{proposition}
	\begin{proof}
		We will prove only the statement for $\ev_{\widetilde X}$, and the case for $\overline{\ev}_{\widetilde X}$ is similar.
		We shall construct a functor quasi-inverse to $\ev_{\widetilde X}$. Let $(M, \epsilon)$ be an object of $\Strat(R, \widetilde R(1))$. We need to associate  an  $\cO_{\prism}$-crystal $M_{\prism}$  to $(M, \epsilon)$.

		Let $(B\to B/ IB\leftarrow R)$ be an object in $(X/A)_{\prism}$, and  $(\widetilde B\to  \widetilde B/I\tilde B\leftarrow R)$ be the product of $(B\to B/ IB\leftarrow R)$ and $\widetilde X$ given by  Lemma~\ref{L:final-cover}. By the  universal property of  prismatic envelopes, the canonical map of bounded prisms $p_B: (\widetilde R, I\widetilde R)\to (\widetilde B, I\widetilde B)$ induces  a commutative diagram  of 2-truncated cosimplicial $\delta$-$A$-algebras: 
		\[\begin{tikzcd}
			\widetilde R \arrow[r, shift left]
			\arrow[r, shift right]
			\arrow [d, "p_B"] 
			& \widetilde R(1) 
			\arrow[r, shift left=2]
			\arrow[r]
			\arrow[r, shift right=2]
			\arrow[d, "p_{B(1)}"]
			& \widetilde R(2)\arrow[d, "p_{B(2)}"]\\
			\widetilde B \arrow[r, shift left] \arrow[r, shift right] &\widetilde B\otimes_B\widetilde B
			\arrow[r, shift left=2]
			\arrow[r]
			\arrow[r, shift right=2]
			&\widetilde B\otimes_B\widetilde B\otimes_B\widetilde B.
		\end{tikzcd}
		\]
		Applying the functor $p_B^*$ to $(M,\epsilon)$, one gets a descent pair (Def.~\ref{D:descent-data}) $(p_{B}^*M, p_{B(1)}^*\epsilon)$ relative to the $(p,I)$-completely faithfully flat map  $B\to \tilde B$  such that $p_{B}^*M$ is $(p,I)$-completely flat over $\widetilde B$. By Proposition~\ref{P:flat-descent}(2),  there exists  a derived $(p, IB)$-complete and $(p,IB)$-completely flat $B$-module $M_B$ such that $M_B\widehat{\otimes}_B\widetilde B\simeq  p_{B}^*(M)$.

		Let $f: (C\to C/IC\leftarrow R)\to (B\to B/IB\leftarrow R)$ be  a morphism  in $(X/A)_{\prism}$, and  $(\widetilde C\to \widetilde C/I\widetilde C\leftarrow R)$ be the product of $(C\to C/IC\leftarrow R)$ with $\widetilde X$. We need to check that there exists a transition isomorphism \[c_f: f^*(M_B)=M_B\widehat{\otimes}_BC\xra{\sim}M_C\]
		satisfying the natural cocyle condition for a composition of morphisms as in Lemma~\ref{L:crystal-data}. 
		Denote by   $\tilde  f: (\widetilde B, I\widetilde B)\to  (\widetilde C, I\widetilde C)$ the  map of bounded prisms induced by $f$. Then one has $p_C=\tilde f\circ p_B.$   By functoriality, one has  a composition  of  isomorphisms
		\[
		\tilde{c}_f\colon M_B\widehat{\otimes}_BC\widehat{\otimes}_C\widetilde C=M_B\widehat{\otimes}_B{\widetilde B}\widehat{\otimes}_{\widetilde B}\widetilde C\simeq \tilde f^*p_B^*(M)=p_C^{*}(M)\simeq M_C\widehat{\otimes}_C\widetilde C;\]
		moreover,  the pullbacks of $\tilde c_f$ via the two canonical maps $\widetilde C\rightrightarrows \widetilde C\widehat\otimes_C\widetilde C$ coincide. Then the desired transition isomorphism $c_f$ is obtained by descent.
		Given  a composition of morphism 
		\[f\circ g: (D\to D/ID\leftarrow R)\xra{g} (C\to C/IC\leftarrow R)\xra{f} (B\to B/IB\leftarrow R)\] 
		in $(X/A)_{\prism}$, one has  the cocycle condition $c_{f\circ g}=c_g\circ g^*(c_f)$. Indeed, this     can be easily checked after base change to $\widetilde D$ by  functoriality, and we conclude by faithfully flat descent (Proposition~\ref{P:flat-descent}). 
		
		Now by Lemma~\ref{L:crystal-data}, the data $(\{M_B\}, \{c_f\})$ is equivalent to an $\cO_{\prism}$-crystal $M_{\prism}$ on $(X/A)_{\prism}$.
		The construction $(M,\epsilon)\mapsto M_{\prism}$ is clearly functorial, which gives a functor that is easily checked to be a quasi-inverse of   $\ev_{\widetilde X}$.
		
	\end{proof}
	
	\begin{remark}
		When $(A,I)$ is a bounded prism over  $(\ZZ_q[[q-1]], ([p]_q))$ (with $\delta(q)=0$ and $[p]_q=\frac{q^p-1}{q-1}$),  Prop.~\ref{P:crystal-stratification} was obtained by \cite[Cor. 6.7]{GSQ} and \cite[Chap. 3]{MS}.
		An analogue  for $q$-crystalline crystals was proved   in \cite[Thm. 1.3.3]{chatz}.
	\end{remark}
	
	We can use the simplicial object $\widetilde X(\bullet)$ to compute the cohomology of an  $\cO_{\prism}$-crystals or an  $\overline{\cO}_{\prism}$-crystal. For an abelian sheaf $\calF$ on $(X/A)_{\prism}$,  the \v{C}ech--Alexander complex $\CA(\widetilde X(\bullet), \calF)$ is defined as the simple complex associated to the cosimplicial abelian group  $\calF(\widetilde X(\bullet))$, i.e. one has 
	\[
	\CA(\widetilde X(\bullet), \calF)= \bigg(\calF(\widetilde X)\to \cdots \to \calF(\widetilde X(n))\xra{d^n}\calF(\widetilde X(n+1))\to \cdots
	\bigg)
	\]
	with 
	$
	d^n=\sum_{i=0}^{n+1} (-1)^i\pr_{i}^{n+1, *} , 
	$
	where $\pr_{i}^{n+1, *} \calF(\widetilde X(n))\to  \calF(\widetilde X(n+1))$ is the map induced by the projection 
	$\pr_{i}^{n+1}: \widetilde X(n+1)\to \widetilde X(n)$ that corresponds to $\delta_i^{n+1}:\widetilde R(n)\to \widetilde R(n+1)$.   If $\calF$ is an $\cO_{\prism}$-crystal,  the projection $\pr^n_{0}: \widetilde X(n)\to \widetilde X$  to the $0$-th copy  induces an isomorphism 
	\[
	c_{\pr^n_0}: \calF(\widetilde X)\widehat\otimes_{\widetilde R}\widetilde R(n)\xra{\sim} \calF(\widetilde X(n)).
	\]
	Then  $\CA(\widetilde X(\bullet), \calF)$ is isomorphic to 
\begin{equation*}
	\calF(\widetilde X)\to \calF(\widetilde X)\widehat\otimes_{\widetilde R}\widetilde R(1)\to \calF(\widetilde X)\widehat\otimes_{\widetilde R}\widetilde R(2)\to\cdots
\end{equation*}
with the $n$-th differential given by 
$$d^n=\sum_{i=0}^{n+1}(-1)^ic_{\pr_0^{n+1}}^{-1}\circ \pr_i^{n+1,*}\circ c_{\pr_0^n} \colon  \calF(\widetilde X)\widehat \otimes_{\widetilde R}\widetilde R(n)\to \calF(\widetilde X)\widehat \otimes_{\widetilde R}\widetilde R(n+1).
 $$

		\begin{proposition}[\v{C}ech--Alexander complex]\label{P:CA-complex}
		
		Let $\calF$ be an $\cO_{{\prism}}$-crystal or   an $\overline\cO_{\prism}$-crystal on $(X/A)_{\prism}$. Then under Assumption~\ref{A:local-lift},  $R\Gamma((X/A)_{\prism}, \calF)$ is computed by the \v{C}ech--Alexander  complex $\CA(\widetilde X(\bullet), \calF)$.

	\end{proposition}
	\begin{remark}
If $\calF$ is an $\overline \cO_{\prism}$-crystal, its \v{C}ech--Alexander complex $\CA(\widetilde X(\bullet), \calF)$ is isomorphic to 
		\begin{equation}\label{E:CA-complex}
		\calF(\widetilde X)\to \calF(\widetilde X)\widehat\otimes_{  R} R(1)\to \calF(\widetilde X)\widehat\otimes_{  R} R(2)\to\cdots \end{equation}
		where the $n$-th differential is given by a similar formula as above. 
	\end{remark}
	
	For the proof of this proposition, we need the following 
	\begin{lemma}\label{L:Vanishing cohomology}
		Let $\calF$ be an $\cO_{{\prism}}$-crystal or $\overline \cO_{{\prism}}$-crystal on $(X/A)_{\prism}$. Then for any object $U=(B\to B/IB\leftarrow R)$ of $(X/A)_{\prism}$ and any integer $q>0$, we have 
		\[H^q(U, \calF)=0.\]
	\end{lemma}
	
	\begin{proof}
		We will  only treat the case when $\calF$ is an  $\cO_{{\prism}}$-crystal, the case of $\overline \cO_{{\prism}}$-crystals being similar. 
		Let $V=(C\to C/IC\leftarrow R)$ be a cover of $U$ in $(X/A)_{\prism}$. Then by the crystal property of $\calF$,  the \v{C}ech complex of $\calF$ for the cover $V\to U$ is identified with the complex associated to the cosimplicial $B$-module $\calF(U)\widehat{\otimes}_{B}C^{\otimes \bullet}$. It follows from Prop.~\ref{P:flat-descent}(1) that its higher cohomology groups vanish. Then we conclude by \cite[Expos\'e V, Cor. 4.3]{SGA4} that $H^q(U, \calF)=0$ for all $q>0$.
	\end{proof}
	
	\begin{proof}[Proof of Prop.~\ref{P:CA-complex}]
		Since $\widetilde X$ is a cover of the final object of the topos $(X/A)_{\prism}^{\sim}$ (Lemma~\ref{L:final-cover}), we have a spectral sequence 
		\[E_1^{i,j}=H^{j}(\widetilde X(i), \calF)\Longrightarrow H^{i+j}((X/A)_{\prism}, \calF).\]
		By Lemma~\ref{L:Vanishing cohomology}, we have $H^j(\widetilde X(i),\calF)=0$ for  $j>0$. Hence $R\Gamma((X/A)_{\prism}, \calF)$ is computed by the \v{C}ech-Alexandre complex 
		$\Gamma(\widetilde X(i),\calF)$, which is isomorphic to \eqref{E:CA-complex} by the crystal property of $\calF$. 
	\end{proof}
	
	\begin{corollary}[weak base change]\label{C:weak base change}
	Let $X$ be a smooth $p$-adic formal scheme over $\Spf(A/I)$ (without assuming Assumption~\ref{A:local-lift}).
	Let $\calF$ be an $\cO_{\prism}$-crystal (resp.   an $\overline \cO_{\prism}$-crystal) on $(X/A)_{\prism}$. 
	Let $(A,I)\to (A',I')$ be a morphism of bounded prisms of finite tor-dimension,  $X'=X\times_{\Spf(A/I)}\Spf(A'/I')$ and $\calF'$ be the pullback of $\calF$ to $(X'/A')_{\prism}$. 
	Then the natural completed base change map 
	\begin{align*}&(R\nu_{X/A,*}(\calF)\otimes^L_AA')^{\wedge}\xra{\sim} R\nu_{X'/A',*}( \calF') \\
		(resp. \quad \quad 
		&(R\nu_{X/A,*}( \calF)\otimes^L_{A/I}A'/I')^{\wedge}\xra{\sim} R\nu_{X'/A',*}(\calF'))
	\end{align*}
	is  an isomorphism.
\end{corollary}
\begin{proof}
	The problem is clearly local for the \'etale topology of $X$. Up to \'etale localization, we may impose thus Assumption~\ref{A:local-lift}. 
	In this case, the statement  follows immediately from Prop.~\ref{P:CA-complex} and the fact that the formation of \v{C}ech--Alexander complex \eqref{E:CA-complex} commutes with the base change  $A\to A'$.
\end{proof}

\begin{remark}
	After establishing Theorem~\ref{T:cohomology-Higgs}, we will see that Corollary~\ref{C:weak base change} holds without the assumption that   $A\to A'$ is of finite tor-dimension (cf. Theorem~\ref{T:Obar-crystal}(2)). 
\end{remark}

	\section{Prismatic Crystals and Higgs Fields }

	In this section, we keep the notation of Section~\ref{S:local-crystal}. 
	We will restrict ourselves to   the following  special case of Assumption~\ref{A:local-lift}.
	\begin{situation}\label{A:local-coordinate}
		%There is an important special case of Assumption~\ref{A:local-lift} is satisfied.  
		We assume that $X=\Spf(R)$  admits an $(p,I)$-completely \'etale map to $ \Spf(A/I\langle \underline{T}_n\rangle)$, where  $A/I\langle \underline{T}_n\rangle:=A/I\langle {T}_1,\dots, T_n\rangle$ denotes the convergent power series ring over $A/I$  in $n$ variables (for the $p$-adic topology).  Then by deformation theory, there exists a unique derived $(p,I)$-complete and  $(p,I)$-completely \'etale $A\langle 
		\underline{T}_n\rangle$-algebra $\widetilde R$ which makes the following diagram cocartesian: 
		\[
		\xymatrix{\widetilde R\ar[r] & R\\
			A\langle \underline{T}_n\rangle\ar[u]\ar[r] &A/I\langle \underline{T}_n\rangle\ar[u].
		}
		\]
		We  choose a $\delta$-structure on $A\langle \underline{T}_n\rangle$ extending that on $A$.  Then by \cite[Lemma~2.18]{BS}, it   extends uniquely to a $\delta$-structure on $\widetilde R$. In particular,  Assumption~\ref{A:local-lift} is satisfied.  For technical reasons, we  suppose that  one of the following assumptions is satisfied:
		
		\begin{enumerate}
			\item $(A,I)$ is a crystalline prism, i.e. $I=(p)$;
			
			\item there exists a map of bounded prisms $(A_0,I_0)\to (A,I)$ such that 
			\begin{itemize}
				\item the Frobenius map on $A_0/p$ is  flat,
				
				\item $A_0/I_0$ is $p$-torsion free;
				
				\item the $\delta$-structure on $A\langle \underline{T}_n\rangle$ descends to a $\delta$-structure on $A_0\langle \underline{T}_n\rangle$.
			\end{itemize} 
		\end{enumerate} 
		For instance, if $I=(d)$ is principal and we take the $\delta$-structure with $\delta(T_i)=0$, then assumption (2) is  satisfied with  $(A_0,I_0)=(\ZZ_p\{d, \delta(d)^{-1}\}^{\wedge}, (d))$,  the $(d,p)$-completed universal oriented prism. 
		%such that $\delta(T_i)=0$.\footnote{If either $I=(p)$ or  $A/I$ is $p$-torsion free and the Frobenius map on $A/p$ is flat, then the same kind of arguments also work with any  $\delta$-structure on   $A\langle \underline{T}_n\rangle$ compatible with that on $A$.} 
		
	\end{situation}
	We will  give an explicit description of  $\overline \cO_{\prism}$-crystals on $(X/A)_{\prism}$ in terms of Higgs fields, and compute  the cohomology of $\overline \cO_{\prism}$-crystals via the de Rham cohomology of its associated Higgs field. Recall the  cosimplicial object $\widetilde R(\bullet)$ defined in \S\ref{S:cosimplicial object}. 
	We give first a more transparent description of   $\widetilde {R}(1)$. 
	Consider the canonical diagonal   surjection 
	\[A\langle \underline{T}_n\rangle \widehat{\otimes}_AA\langle \underline{T}_n\rangle \to A/I \langle \underline{T}_n\rangle,\] and denote its kernel by $J_n$. Note that we have an isomorphism 
	\[
	A\langle \underline{T}_n\rangle \widehat{\otimes}_AA\langle \underline{T}_n\rangle\simeq A\langle \underline{T}_n\rangle\langle \xi_1, \dots, \xi_n\rangle, 
	\]
	given by $T_i\otimes 1\mapsto T_i$ and $1\otimes T_i\mapsto T_i+\xi_i$ so that $\xi_i$ corresponds to $1\otimes T_i-T_i\otimes 1$ for $1\leqslant i\leqslant n$. 
	Via this isomorphism, the ideal $J_n$ corresponds to $(I, \xi_1, \dots, \xi_n)$. 
	Recall that we have chosen a $\delta$-structure on $A\langle \underline T_n\rangle $ compatible with that on $A$. We equip $A\langle \underline{T}_n\rangle\langle \xi_1, \dots, \xi_n\rangle$ with the $\delta$-structure that corresponds to the canonical induced  tensor $\delta$-structure on $A\langle \underline{T}_n\rangle \widehat{\otimes}_AA\langle \underline{T}_n\rangle$:
	explicitly, if $\delta(T_i)=f_i(\underline{T})\in A\langle \underline T_n\rangle$, we have 
	\begin{align}\label{E:delta-xi}
		\delta(\xi_i)=\sum_{j=1}^{p-1}\frac{1}{p}\binom{p}{j} \xi_i^jT_i^{p-j}+f_i(\underline{T}+\underline{\xi})-f_i(\underline{T})\in (\xi_1,\dots, \xi_n)
	\end{align}
	Applying the construction of \cite[Lemma 3.13]{BS}, we get the prismatic envelope 
	\begin{equation}\label{E:basic case}
		\big(A\langle \underline{T}_n\rangle \widehat{\otimes}_AA\langle \underline{T}_n\rangle\big)\{\frac{J_n}{I}\}^{\wedge}\simeq A\langle \underline {T}_n\rangle\{ \frac{\xi_1}{I}, \dots, \frac{\xi_n}{I}\}^{\wedge},
	\end{equation}
	which is $(p,I)$-completely faithfully flat over $A\langle \underline{T}_n\rangle $. The  formation of $A\langle \underline {T}_n\rangle\{ \frac{\xi_1}{I}, \dots, \frac{\xi_n}{I}\}^{\wedge}$ commutes with base change in $(A\langle \underline{T}_n\rangle, I A\langle \underline{T}_n\rangle)$ in the following sense:  
	If $(C,IC)$ is  a bounded prism  over  $(A\langle \underline{T}_n\rangle, I A\langle \underline{T}_n\rangle)$ and we extend the $\delta$-structure on $C$ to $C\langle \xi_1, \dots, \xi_n\rangle$ with  $\delta(\xi_i)$ given by the image of  \eqref{E:delta-xi} in $C\langle \xi_1, \dots, \xi_n\rangle$, then 
	\[C\{\frac{\xi_1}{I}, \dots, \frac{\xi_n}{I}\}^{\wedge}:=
	C\widehat{\otimes}_{A\langle \underline{T}_n\rangle} A\langle \underline{T}_n\rangle\{\frac{\xi_1}{I}, \dots, \frac{\xi_n}{I}\}^{\wedge}\]
	is nothing but the prismatic envelope of  $C\langle \xi_1, \dots, \xi_n\rangle$  with respect to the ideal $(I, \xi_1, \dots, \xi_n)$. 
	Note that the  surjection $C\langle \xi_1, \cdots, \xi_n\rangle \to C$ sending all $\xi_i$ to $0$ induces a canonical surjection  of bounded prisms:  
	\begin{equation}\label{E:augmentation map}
	C\{\frac{\xi_1}{I}, \dots, \frac{\xi_n}{I}\}^{\wedge} \to C.
	\end{equation}
	%Note that there exists a canonical map $$C/IC\to C\{\frac{\xi_1}{I}, \dots, \frac{\xi_n}{I}\}^{\wedge}/IC\{\frac{\xi_1}{I}, \dots, \frac{\xi_n}{I}\}^{\wedge}$$
	%obtained by extension of  scalars from the corresponding  map for $C=A\langle \underline T_n\rangle$.

	By \eqref{E:basic case}  and the functoriality of prismatic envelope, there exists a canonical map of $\delta$-$A$-algebras
	\begin{equation}\label{E:formula-1}
		A\langle \underline{T}_n\rangle\{\frac{\xi_1}{I}, \dots, \frac{\xi_n}{I}\}^{\wedge}\to \widetilde {R}(1).\end{equation}
	Recall the morphism $\delta_1^{1}: \widetilde R\to \widetilde R(1)$ induced by $\widetilde R\to \widetilde R\widehat\otimes_A\widetilde R\colon b\mapsto b\otimes 1$ (see \eqref{E:simplicial-projection}). 
	Taking the tensor product  of  $\delta^1_1$ with \eqref{E:formula-1},  one gets a map
	\[\eta: \widetilde R\{\frac{\xi_1}{I}, \dots, \frac{\xi_n}{I}\}^{\wedge}=\widetilde R\widehat{\otimes}_{A\langle \underline T_n\rangle}A\langle \underline T_n\rangle\{\frac{\xi_1}{I}, \dots, \frac{\xi_n}{I}\}^{\wedge} \to \widetilde R(1) \]
	which induces a map of prisms over $(A,I)$.

	\begin{lemma}\label{L:R-one-structure}
		The morphism $\eta$ induces an isomorphisms of prisms 
		\[(\widetilde R\{\frac{\xi_1}{I}, \dots, \frac{\xi_n}{I}\}^{\wedge}, I\widetilde R\{\frac{\xi_1}{I}, \dots, \frac{\xi_n}{I}\}^{\wedge})\xra{\sim} (\widetilde R(1), I\widetilde R(1)).\]
	\end{lemma}
	
	\begin{proof}
		We need to construct an inverse to $\eta$. Consider the following  diagram 
		\begin{equation}\label{E:commutative-diag}\xymatrix{A\langle \underline T_n\rangle \ar[r]\ar[d]_{T_i\mapsto T_i+\xi_i} & \widetilde R\ar[d]\ar@{-->}[ld]_{i_2}\\
				\widetilde R\{\frac{\xi_1}{I}, \dots, \frac{\xi_n}{I}\}^{\wedge}\ar[r] & \widetilde R\{\frac{\xi_1}{I}, \dots, \frac{\xi_n}{I}\}^{\wedge}/I \widetilde R\{\frac{\xi_1}{I}, \dots, \frac{\xi_n}{I}\}^{\wedge},
		}\end{equation}
		where the right vertical map is the composed  canonical map:
		\[\widetilde R\to \widetilde R/I\widetilde R \to \widetilde R\{\frac{\xi_1}{I}, \dots, \frac{\xi_n}{I}\}^{\wedge}/I \widetilde R\{\frac{\xi_1}{I}, \dots, \frac{\xi_n}{I}\}^{\wedge}.\]
		It is clear that the square of the above diagram is commutative. Since $A\langle \underline T_n\rangle\to \widetilde R$ is $(p, I)$-completely \'etale and $\widetilde R\{\frac{\xi_1}{I}, \dots, \frac{\xi_n}{I}\}^{\wedge}$ is $I$-adically complete, there exists a unique map $i_2: \widetilde R\to \widetilde R\{\frac{\xi_1}{I}, \dots, \frac{\xi_n}{I}\}^{\wedge}$ as the dotted arrow that makes all triangles in the diagram commute. Moreover, since the left vertical arrow  is a map of $\delta$-$A$-algebras and the $\delta$-structure on $\widetilde R$ is uniquely determined by its restriction to $A\langle \underline T_n\rangle$ (cf. \cite[Lemma 2.18]{BS}), it follows that $i_2$ is a map of $\delta$-$A$-algebras.
		Consider the morphism 
		\[f: \widetilde R\widehat{\otimes}_A\widetilde R\to  \widetilde R\{\frac{\xi_1}{I}, \dots, \frac{\xi_n}{I}\}^{\wedge}\]
		given by the tensor product of the natural inclusion $i_1: \widetilde R\to \widetilde R\{\frac{\xi_1}{I}, \dots, \frac{\xi_n}{I}\}^{\wedge} $ and  $i_2$. By the commutativity of the diagram \eqref{E:commutative-diag}, $i_1$ and $i_2$ agree after post-composition  with the natural surjection 
		\[
		\widetilde R\{\frac{\xi_1}{I}, \dots, \frac{\xi_n}{I}\}^{\wedge} \to \widetilde R\{\frac{\xi_1}{I}, \dots, \frac{\xi_n}{I}\}^{\wedge}/I \widetilde R\{\frac{\xi_1}{I}, \dots, \frac{\xi_n}{I}\}^{\wedge}.
		\]
		Therefore,  the composed map 
		\[
		\widetilde R\widehat{\otimes}_A\widetilde R\xra{f}  \widetilde R\{\frac{\xi_1}{I}, \dots, \frac{\xi_n}{I}\}^{\wedge}\to  \widetilde R\{\frac{\xi_1}{I}, \dots, \frac{\xi_n}{I}\}^{\wedge}/I  \widetilde R\{\frac{\xi_1}{I}, \dots, \frac{\xi_n}{I}\}^{\wedge}
		\]
		factors though $ \widetilde R\widehat{\otimes}_A\widetilde R\to \widetilde R$, i.e. $f$ sends $\widetilde J(1)\subset \widetilde R\widehat{\otimes}_A\widetilde R$ to $I \widetilde R\{\frac{\xi_1}{I}, \dots, \frac{\xi_n}{I}\}^{\wedge}$. 
		By the universal property of the prismatic envelope, it induces  a morphism of prisms over $(A,I)$: 
		\[\eta': (\widetilde R(1), I\widetilde R(1))\to (\widetilde R\{\frac{\xi_1}{I}, \dots, \frac{\xi_n}{I}\}^{\wedge}, I\widetilde R\{\frac{\xi_1}{I}, \dots, \frac{\xi_n}{I}\}^{\wedge}).\]
		It remains  to check that  $\eta$ and $\eta'$ are inverse of each other. If we regard   $\widetilde R\{\frac{\xi_1}{I}, \dots, \frac{\xi_n}{I}\}^{\wedge}$  and $\widetilde R(1)$ as $\widetilde R$-algebras via $i_1$ and $\delta^1_1$ respectively, then both $\eta$ and $\eta'$ are  maps of $\widetilde R$-algebras. In order to show that $\eta'\circ \eta=\id$, it suffices to see that $\eta'\circ\eta(\xi_i)=\xi_i$. For this, one can reduce to the case $\widetilde R=A\langle \underline{T}_n\rangle $, which follows from \eqref{E:basic case}.  
		To check $\eta\circ \eta'=\id$, it is enough, by the universal property of $(\widetilde R(1), I\widetilde R(1))$,  to show that $\eta\circ\eta'\circ\delta^1_i=\delta^1_i$ for $i=0,1$ where  $\delta^1_0, \delta^1_1: \widetilde R\to \widetilde R(1)$ are the two natural maps defined in  \eqref{E:simplicial-projection}. For $i=1$, this is evident. For $i=0$, by noting  that $i_2=\eta'\circ \delta^1_0$ by the definition  of $\eta'$, we are reduced to proving $\eta\circ i_2=\delta_0^1$. By the commutative diagram \eqref{E:commutative-diag}, both sides agree after restricting to $A\langle \underline T_n\rangle$ or after reducing modulo $I$. We conclude by the fact that $A\langle \underline T_n \rangle \to \widetilde R$ is $(p,I)$-completely \'etale. 
	\end{proof}

From now on, we will use Lemma~\ref{L:R-one-structure} to identify $\widetilde R(1)$ with $\widetilde R\{\frac{\xi_1}{d},\cdots, \frac{\xi_n}{I}\}$. 
Let $\widetilde \calK\subset \widetilde R(1)$ be the kernel of the surjection map 
\[\widetilde R(1)=\widetilde R\{\frac{\xi_1}{I},\dots, \frac{\xi_n}{I}\}^{\wedge}\to \widetilde R,\] 
defined in \eqref{E:augmentation map}.  The following technical Lemma will play an important role in our later discussion. 

\begin{lemma}\label{L:delta-phi}
	Let $\phi$ denote the Frobenius map on $\widetilde R(1)$. Then for any $x\in \widetilde \calK$, we have $\phi(x)\in I\widetilde R(1)$. 
\end{lemma}

\begin{proof}
	Let $ A'$ be a $(p,I)$-completely faithfully flat map of $\delta$-$A$-algebra such that $IA'$ is principal.  Put $\widetilde R'=\widetilde R\widehat{\otimes}_AA'$. Then  one has $\widetilde R(1)\widehat{\otimes}_AA'\cong \widetilde R'(1)$.  As 
	$\widetilde R(1)/I\widetilde R(1)\to \widetilde R'(1)/I\widetilde R'(1)$  is injective, it suffices to show that $\phi(x)\in I\widetilde R'(1)$ for all $x$ in the kernel of the canonical surjection map $\widetilde R'(1)\to \widetilde R'$. Therefore, up to a base change in $(A,I)$, we may  assume that $I$ is generated by  a distinguished element $d\in A$. Then  we have
	\[\widetilde R(1)\cong \widetilde R \{\frac{\xi_1}{d}, \dots, \frac{\xi_n}{d}\}^{\wedge}= \bigg(\widetilde R\langle \xi_1,\dots, \xi_n\rangle \{X_1,\cdots, X_n\}/(\xi_i-dX_i:1\leqslant i\leqslant n)_{\delta}\bigg)^{\wedge}
	\] 
	where $\widetilde R\langle \xi_1,\dots, \xi_n\rangle \{X_1,\cdots, X_n\}$ is the free $\delta$-algebra over $\widetilde{R}\langle \xi_1,\dots, \xi_n\rangle $ in $n$-variables, and $(\xi_i-dX_i:1\leqslant i\leqslant n)_{\delta}$ is the  ideal generated by $\delta^r(\xi_i-dX_i)$  for all $r\geqslant 0$ and  $1\leqslant i\leqslant n$.
	Let  $x_i$ denote the image of $X_i$ in $\widetilde R(1)$. Then  $\widetilde \calK\subset \widetilde R(1)$ is 
	the closure of the ideal generated by $\delta^{r}(x_i)$ with $r\geqslant 0 $ and $1\leqslant i\leqslant n$.  Therefore, in order to prove the Lemma, it suffices to see $\phi(\delta^r(x_i))\in I\widetilde R(1)$ for all $r\geqslant 0$ and $1\leqslant i\leqslant n$.

	 Let $\calJ_r\subseteq \widetilde R(1)$ denote the closed ideal generated by $d^{p^j}\delta^j(x_k)$ with $0\leqslant j\leqslant r$ and $1\leqslant k\leqslant n$.
	We claim that   for every  $r\geqslant 0$ and $1\leqslant i\leqslant n$, there exists a $b_{i,r+1}\in \widetilde R(1)^{\times}$ such that 
	\begin{equation*}\label{E:delta-phi}
		\phi(\delta^r(x_i))=\delta^r(\phi(x_i))=\delta^r(x_i)^p+p\delta^{r+1}(x_i)\equiv  b_{i, r+1}d^{p^{r+1}}\delta^{r+1}(x_i)\mod \calJ_r.
	\end{equation*}
This claim  implies immediately $\phi(\delta^r(x_i))\in I\widetilde R(1)$. To show the claim,   we proceed  by induction on $r\geqslant0$. For $1\leqslant i\leqslant n$, we have
\[\delta(\xi_i)=\delta(d x_i)=d^p\delta(x_i)+\phi(x_i)\delta(d).\]
In view of \eqref{E:delta-xi} and $\xi_j=dx_j$,  we  have $\delta(\xi_i)\in \calJ_0$.  As $d$ is distinguished, we have $\delta(d)\in A^{\times}$ and 
\[
\phi(x_i)\equiv -\frac{1}{\delta(d)}d^p\delta(x_i)\mod \calJ_0.
\]
This proves the claim  for $r=0$. Suppose now  that the statement  holds for all integers $\leqslant r$ so that    there exist $b_{i,j}^{r,k}\in \widetilde R(1)$ and $b_{i,r+1}\in \widetilde R(1)^{\times}$ such that
\[
\delta^r(\phi(x_i))=\sum_{j=0}^r\sum_{k=1}^nb^{r,k}_{i,j}d^{p^j}\delta^j(x_k)+b_{i,r+1} d^{p^{r+1}}\delta^{r+1}(x_i).
\]
Applying $\delta$, we get 
\begin{align*}
	\delta^{r+1}(\phi(x_i))&=\delta\bigg(\sum_{j=0}^r\sum_{k=1}^nb^{r,k}_{i,j}d^{p^j}\delta^j(x_k)+b_{i,r+1} d^{p^{r+1}}\delta^{r+1}(x_i)\bigg)\\
	&\equiv \sum_{j=0}^r\sum_{k=1}^n\delta(b_{i,j}^{r,k}d^{p^j}\delta^j(x_k))+\delta(b_{i,r+1}d^{p^{r+1}}\delta^{r+1}(x_i)) \mod \calJ_{r+1}.
\end{align*}
Note that   $$\delta(b_{i,j}^{r,k}d^{p^j}\delta^j(x_k))=\delta(b_{i,j}^{r,k}d^{p^j})\phi(\delta^j(x_k))+(b_{i,j}^{r,k})^p d^{p^{j+1}}\delta^{j+1}(x_k),$$ and  $\phi(\delta^j(x_k))=\delta^j(\phi(x_k))\in \calJ_j\subset \calJ_{r+1}$ by the induction hypothesis. We have  also 
\[\delta(b_{i,r+1}d^{p^{r+1}}\delta^{r+1}(x_k))=\delta(b_{i,r+1}d^{p^{r+1}}) \phi(\delta^{r+1}(x_i))+ b_{i,r+1}^p d^{p^{r+2}}\delta^{r+2}(x_i). \]
It follows that 
\[(1-\delta(b_{i,r+1}d^{p^{r+1}}))\delta^{r+1}(\phi(x_i))\equiv b_{i,r+1}^p d^{p^{r+2}}\delta^{r+2}(x_i)\mod \calJ_{r+1}.\]
Now by induction on $m\geqslant 1$, it is easy to see  that $\delta(bd^{m})\in (d^p,p)^{m-1}$ for any $b\in \widetilde R(1)$. It follows that $\delta(b_{i,r+1}d^{p^{r+1}})$ is topologically nilpotent  and $1-\delta(b_{i,r+1}d^{p^{r+1}})$ is invertible as $\widetilde R(1)$ is $(p,d)$-complete. We conclude that 
\[\delta^{r+1}(\phi(x_i))\equiv b_{i,r+2}d^{p^{p+2}}\delta^{r+2}(x_i)\mod \calJ_{r+1}\]
with $b_{i,r+2}=(1-\delta(b_{i,r+1}d^{p^{r+1}}))^{-1}b_{i,r+1}^p\in \widetilde R(1)^{\times}$. This finishes the induction process and hence the proof of Lemma~\ref{L:delta-phi} is finished.
\end{proof}

	\subsection{A divided power algebra} Recall that $\widetilde J(1)$ is the kernel of the canonical diagonal surjection $\widetilde R\widehat \otimes_{A}\widetilde R\to \widetilde R$. Since $\widetilde R$ is $(p, I)$-completely \'etale over $A\langle \underline T_n\rangle $,  the module of continuous differential $1$-forms of $\widetilde R$ relative to $A$ 
\[\Omega^1_{\widetilde R/A}:=\widetilde J(1)/\widetilde J(1)^2\] 
is free of rank $n$ over $\widetilde R$ with basis $(\rd T_i: 1\leqslant i\leqslant n)$.  
Let $\Bl_{\calJ}(\widetilde R\widehat \otimes_{A}\widetilde R)$ denote the  blow-up of $\Spec(\widetilde R\widehat \otimes_{A}\widetilde R)$ along the ideal  $\calJ:=(I, \tilde J(1))$.
Let $\widetilde U$ be the $(p, I)$-adic completion of the affine open subset of $\Bl_{\calJ}(\widetilde R\widehat \otimes_{A}\widetilde R)$, where the inverse image of  $I$ generates that  of $\calJ$, and put $U=\widetilde U\otimes_{A}A/I$. 
Note that  $\calJ\widetilde R(1) $ is generated by $I$.  
The universal property of  blow-up implies that there exists a  morphism of formal schemes $ \Spf(\widetilde R(1))\to \widetilde U$  such that the following diagram is commutative: 
\begin{equation}\label{E:blow-up-diag}
\xymatrix{\Spf(R(1))\ar@{^(->}[r]\ar[d] & \Spf(\widetilde R(1))\ar[d]\\
	U\ar@{^(->}[r]\ar[d] & \widetilde U\ar[d]\\
	\Spf(R)\ar@{^(->}[r] & \Spf(\widetilde R\widehat\otimes \widetilde R).}
\end{equation}
Here, the  bottom horizontal  arrow is the diagonal closed embedding defined by the ideal $\calJ$, and the upper square is cartesian. 

	We can give  a more transparent description of $U$.  Put $\Omega^1_{R}:=\Omega^1_{ \widetilde R/A}\otimes_A A/I$. In general, for an $A/I$-module $M$ and an integer $i\geqslant 0$, we put $M\{i\}=M\otimes_{A/I}(I/I^2)^{\otimes i}$ and $M\{-i\}=M\otimes_{A/I}(I/I^2)^{\vee, \otimes i}$. 

\begin{lemma}\label{L:description-U}
	Let $R\langle \Omega^1_R\{-1\}\rangle  $ denote  the $p$-adic completion of the symmetric power algebra of the  locally free $R$-module $\Omega^1_R\{-1\}$. 
	Then one has 
	\[U\cong \Spf\big(R\langle \Omega^1_R\{-1\}\rangle  \big)
	\]
\end{lemma}
\begin{proof}
	Since $\calJ=(I, \widetilde J(1))$ is locally generated by a regular sequence, $\calJ/\calJ^2$ is  a $R$-module locally free of rank $n+1$, and the exceptional divisor of  $\Bl_{\calJ}(\widetilde R\widehat{\otimes}\widetilde R)$  is isomoprhic to the projective space  $$\bfP(\calJ/\calJ^2):=\mathrm{Proj} 
	\big(R[\calJ/\calJ^2]\big)
	$$ 
	over $\Spec(R)$,
	where $R [\calJ/\calJ^2]$ denotes the symmetric power algebra of $\calJ/\calJ^2$ over $R$. Then $U$ is identified with the  $p$-adic completion of the affine open  subscheme of  $\bfP(\calJ/\calJ^2)$ where the linear function given by  $IR \subset \calJ/\calJ^2$ is invertible. Therefore, the coordinate ring of $U$ is canonically isomorphic to the $p$-adic completion of the symmetric power algebra of the $R$-module 
	\[
	\Hom_R(IR, \widetilde J(1)/\widetilde J(1)^2\otimes_{A}A/I)\cong \Omega^1_R\{-1\}. 
	\]
\end{proof}
Let $\calI^{+}:=\Omega^1_R\{-1\}R\langle \Omega^1_R\{-1\}\rangle $ be the kernel of the canonical surjection $R\langle \Omega^1_R\{-1\}\rangle   \to R$. 
Let $(R^{PD}(1), \calK)$ denote the $p$-adic completion of  the  PD-envelope with respect to $\calI^+$, and  $\calK^{[2]}\subset R^{PD}(1)$ be the closed ideal generated by elements $x^{[n]}$ for $x\in \calK$ and $n\geqslant 2$. Then the   canonical map $\Omega^1_R\{-1\}\subset \calI^+ \to \calK$ induces an isomorphism 
\begin{equation}\label{E:differential-isomorphism}
	\Omega^1_R\{-1\} \cong   \calK/\calK^{[2]}.
\end{equation}
As $\rd T_i\in \Omega^1_R$ is nothing but the image of $\xi_i=1\otimes T_i-T_i\otimes 1$, it is convenient to write 
\begin{equation}\label{E:R-PD-one}
R\langle \Omega^1_R\{-1\}\rangle =R\langle \frac{\xi_1}{I}, \cdots, \frac{\xi_n}{I}\rangle,\quad R^{PD}(1)= R\{\frac{\xi_1}{I}, \dots, \frac{\xi_n}{I}\}^{PD,\wedge}.
\end{equation}
This notation makes sense, because if $I=(d)$ is principal,  then $R\langle \Omega^1_R\{-1\}\rangle $ and $R^{PD}(1)$  are respectively the convergent power series ring and  the $p$-completed divided power polynomial ring in $n$-variables $\{\frac{\xi_i}{d}: 1\leqslant  i\leqslant n\}$ over $R$. 

\if false
	
	We now compare $R(1)=\widetilde R(1)/I\widetilde R(1)$ with some divided power polynomial algebra. 
	%Let $\Bl_{J}(\AAA_{\widetilde R}^n)$ denote the  blow-up of $\Spec(\widetilde R[ \xi_1, \dots, \xi_n] )$ along the ideal $J=(I, \xi_1,\dots, \xi_n)$.
	%Let $U$ be the  affine open subset of $\Bl_{J}(\AAA_{\widetilde R}^n)$ where the inverse image of  $I$ generates the inverse image of $(I, \xi_1,\dots, \xi_n)$. 
	Let $ \Omega^1_{ R}$ be the module of continuous differential $1$-forms of $ R$ relative to $A/I$. Since $ R$ is $p$-completely \'etale over $(A/I)\langle \underline T_n\rangle $, $ \Omega^1_{ R}$ is a free $ R$-module of rank $n$ with basis $(\rd T_i: 1\leqslant  i\leqslant n)$. 
	 In general, for an $A/I$-module $N$ and an integer $i\geqslant 0$, we put $N\{i\}=N\otimes_{A/I}(I/I^2)^{\otimes i}$ and $N\{-i\}=N\otimes_R (I/I^2)^{\vee,\otimes i}$. 
	Let $\Sym_{ R}(\Omega^1_{ R}\{-1\} )$ be the symmetric algebra of the locally free $ R$-module $\Omega^1_{R}\{-1\}$. As $\rd T_i\in \Omega_{\widetilde R/A}^1$ is nothing but  the image of $\xi_i=1\otimes T_i-T_i\otimes 1$, it is convenient to write   
	$  R[ \frac{\xi_1}{I}, \dots, \frac{\xi_n}{I}]$  for  $\Sym_{ R}(\Omega^1_{ R}\{-1\})$. 
	This notation makes sense, since if $I=(d)$ is principal, then $\Omega^1_R\{-1\}$ is free of rank $n$ over $R$ with basis $(\rd T_i\otimes d^{-1}, 1\leqslant i\leqslant d)$,  and $\Sym_{R}(\Omega^1_R\{-1\})$ is isomorphic to the polynomial ring  over $R$  in $n$-variables $\frac{\xi_1}{d} ,\cdots ,\frac{\xi_n}{d}$. 
	
	Let $\calI^+ $ be the kernel of the canonical surjection 
	$R[ \frac{\xi_1}{I}, \dots, \frac{\xi_n}{I}]\to R
	$
	sending all positive symmetric powers of $\Omega^1_R\{-1\} $ to $0$. 
	 Let $(R^{PD}(1), \calK)$ denote the $p$-adic completion of the divided power envelope of   $R[ \frac{\xi_1}{I}, \dots, \frac{\xi_n}{I}]$  with respect to the ideal $\calI^+$. By definition,  there exists a divided power structure on $\calK$, and one has $R^{PD}(1)/\calK\cong R$. 
	  Let $\calK^{[2]}\subset R^{PD}(1)$ be the closed ideal generated by $x^{[m]}$ for all $x\in \calK$ and $m\geqslant 2$. Then the  canonical inclusion $\Omega^1_R\{-1\}\subset \calK$ induces an isomorphism 
	  \begin{equation}\label{E:differential-isomorphism}
	 \Omega^1_R\{-1\} \cong   \calK/\calK^{[2]}.
	  \end{equation}
  If $I=(d)$ is principal, then $R^{PD}(1)=R\{\frac{\xi_1}{d}, \cdots, \frac{\xi_n}{d}\}^{PD,\wedge}$ is the $p$-adic completion of the free divided power polynomial ring in the $n$ variables $\frac{\xi_1}{d} ,\cdots ,\frac{\xi_n}{d}$.
  
  \fi
  
\subsection{} We now compare $R(1)=\widetilde R(1)/I\widetilde R(1)$ with  $R^{PD}(1)$. Recall that $\widetilde \calK\subset \widetilde R(1)$ is the canonical projection $\widetilde R(1)\cong \widetilde R\{\frac{\xi_1}{I}, \cdots, \frac{\xi_n}{I}\}^\wedge\to \widetilde R$. 
	 For  $x\in\widetilde  \calK$, we put $\gamma_p(x):=-\frac{\delta(x)}{(p-1)!}\in \widetilde \calK$. We claim  that  
	\begin{enumerate}
		\item $p!\gamma_p(x)\equiv x^p\mod I\widetilde R(1)$ for $x\in \widetilde \calK$;
		\item $\gamma_p(ax)\equiv a^p\gamma_p(x) \mod I\widetilde R(1)$ for any $a\in \widetilde R(1)$ and $x\in \widetilde \calK$;
		
		\item $\gamma_p(x+y)\equiv\gamma_p(x)+\gamma_p(y)+\sum_{i=1}^{p-1}\frac{1}{i!(p-i)!}x^iy^{p-i}\mod I\widetilde R(1)$ for $x,y\in \widetilde \calK$.
x	\end{enumerate}
Indeed, Lemma~\ref{L:delta-phi} says that 
\[\phi(x)=x^p+p\delta(x)\in I\widetilde R(1)\]
for all $x\in \widetilde \calK$, which implies immediately statement (1). Statement (2) follows from the computation:
\[\gamma_p(ax)=a^p\gamma_p(x)+\gamma_p(a)\phi(x)\equiv a^p\gamma_p(x)\mod I\widetilde R(1),
\]
and  (3) is a direct consequence of the additive property of $\delta$. 

	Let $ {\calK'}$ denote the image of $\widetilde\calK$ in $R(1)=\widetilde R(1)/I\widetilde R(1)$. For $\bar x\in {\calK'}$, we denote by  $\gamma_p(\bar x)\in {\calK'} $  the image of $\gamma_p(x)$ for an lift $x\in \widetilde\calK$ of $\bar x$. As $\widetilde\calK\cap I\widetilde R(1)=I\widetilde\calK$, it is easy to see that  $\gamma_p(\bar x)$ does not depend on the choice of $x$.  By  the properties of $\gamma_p$ and \cite[\href{https://stacks.math.columbia.edu/tag/07GS}{Tag 07GS}]{stacks-project}, there exists a unique PD-structure on the ideal $ {\calK'}\subset  R(1)$ such that the $p$-th divided power function is given by $\gamma_p:{\calK'}\to  {\calK'}$. 
	Note that the morphism $\Spf(R(1))\to U$ in \eqref{E:blow-up-diag} gives a map of $R$-algebras
\[
R\langle \Omega^1_{R}\{-1\}\rangle \to R(1)
\]	
that sends the ideal $\calI^+$ to $\calK'$. 
	By the universal property of the PD-envelope, such a map extends uniquely to  a  morphism of  PD-algebras over $R$
\[\psi: R^{PD}(1)\to  R(1)\]
sending $\calK'$ to $\calK$.

	\begin{proposition}\label{P:PD-delta}
		Suppose that we are in Situation~\ref{A:local-coordinate}. Then $\psi$ is an isomorphism of PD-algebras over $R$ that identifies $\calK'$ with $\calK$. 
		%	which is independent of the choice of the $\delta$-structure as in Situation~\ref{A:local-coordinate}.
		%	where  $R\{X_1,\dots, X_n\}^{PD}$ is the $p$-adic completion of the divided power polynomial over  $R$ in $n$-variables. 
		Moreover, via this isomorphism, the two maps $\bar \delta_{i}^1: R\to  R(1)$ with $i=0,1 $, which are reduction mod $I$ of \eqref{E:simplicial-projection},  are both identified with the natural inclusion $R\to R^{PD}(1)$.
	\end{proposition}
	\begin{proof}

		We have to show that $\psi$ is an isomorphism. It suffices to do this after a $(p,I)$-completely faithfully flat base change in $(A,I)$. As the formation of $\psi$ commutes with such a base change, up to changing notation, we may assume that $I=(d)$ is principal. 
		In this case, we have 
		\[\widetilde R(1)\cong \widetilde R \{\frac{\xi_1}{d}, \dots, \frac{\xi_n}{d}\}^{\wedge},\]
		and $R^{PD}(1)$ is the $p$-completed  free divided power polynomial ring in the variables $\{ \frac{\xi_i}{d}: 1\leqslant i\leqslant n\}$. Since $R(1)=\widetilde R(1)/I \widetilde R(1)$ is topologically generated by the image of $\frac{\xi_i}{d}$ and the iterations of their $\gamma_p$'s, we see that $\psi$ is surjective.  It remains to see that $\psi$ is injective.  
		
		We consider first the special case   $(d)=(p)$, i.e. $(A,I)=(A,(p))$ is a crystalline prism. Let $\widetilde R\langle \frac{\xi_1}{p}, \dots, \frac{\xi_n}{p}\rangle$ be the convergent power series ring over $\widetilde R$ in the variables $\{\frac{\xi_i}{d}:1\leqslant i\leqslant n\}$, and  $\widetilde R\{ \frac{\xi_1}{p}, \dots, \frac{\xi_n}{p}\}^{PD,\wedge}$ be the corresponding $p$-completed divided power polynomial ring. 
		One considers the following Frobenius structure on $\widetilde R\langle \frac{\xi_1}{p}, \dots, \frac{\xi_n}{p}\rangle$ given by 
		\[\phi(\frac{\xi_i}{p})=\frac{\phi(\xi_i)}{p}=\frac{\xi_i^p}{p}+\delta(\xi_i)=p^{p-1}\big(\frac{\xi_i}{p}\big)^{p}+\delta(\xi_i),\]
		with $\delta(\xi_i)$ given by \eqref{E:delta-xi}. 
		Note that $\delta(\xi_i)$ is divisible by $p$ in $\widetilde R\langle \frac{\xi_1}{p}, \dots, \frac{\xi_n}{p}\rangle$ so that  $\phi(\frac{\xi_i}{p})\in p\widetilde R\langle \frac{\xi_1}{p}, \dots, \frac{\xi_n}{p}\rangle$. It follows that such a lift of Frobenius, or equivalently such a $\delta$-structure, extends uniquely to $\widetilde R\{ \frac{\xi_1}{p}, \dots, \frac{\xi_n}{p}\}^{PD,\wedge}$: Indeed, if $\phi(\frac{\xi_i}{p})=py_i$, then we put 
		\[
		\phi((\frac{\xi_i}{p})^{[n]})=\frac{p^n}{n!}y_i^{n},
		\]
		which defines a unique Frobenius structure on $\widetilde R\{ \frac{\xi_1}{p}, \dots, \frac{\xi_n}{p}\}^{PD,\wedge}$. Here,  $(\frac{\xi_i}{p})^{[n]}$ denotes the $n$-th divided power of $\frac{\xi_i}{p}$. By the universal property of the prismatic envelope $\widetilde R(1)\cong \widetilde R\{ \frac{\xi_1}{p},\cdots,\frac{\xi_n}{p}\}^{\wedge}$, there exists a unique map of $\delta$-$\widetilde R$-algebras $\widetilde R(1)\to \widetilde R\{\frac{\xi_1}{p},\dots, \frac{\xi_n}{p}\}^{PD,\wedge}.$ Reducing modulo $p$, one gets a map of divided power $R$-algebras
		\[R(1)=\widetilde R(1)/p\widetilde R(1)\to R^{PD}(1)=R\{\frac{\xi_1}{d},\dots, \frac{\xi_n}{d}\}^{PD,\wedge}\]
		which is easily seen to be the inverse of $\psi$.

		We consider now the general case of $d$.  First, we  note that the formation  of  $\psi$ commutes with the \'etale localization in $R$. Therefore, one can reduce to the case  $R=A\langle \underline{T}_n\rangle$.
		By our assumptions  in  Situation~\ref{A:local-coordinate}, there exists a map of bounded prisms $(A_0,I_0)\to (A,I)$ such that the  $\delta$-structure on $A\langle \underline{T}_n\rangle$ descends to $A_0\langle \underline{T}_n\rangle$ and 
		\begin{enumerate}
			\item[(a)] the Frobenius map on $A_0/p$ is flat;
			\item [(b)] $A_0/I_0$ is $p$-torsion free. 
		\end{enumerate}
		Since the formation of $\psi$ commutes with base change in $(A_0,I_0)$, it suffices to prove the statement for $(A_0,I_0)$. 
		Up to  changing notation, we may assume that $(A,I)$ satisfies conditions (a) and (b). 
		
		Consider the ring $B=A\{\frac{\phi(d)}{p}\}^{\wedge}$, the $p$-completed $\delta$-$A$-algebra obtained  by freely adjoining $\frac{\phi(d)}{p}$ to $A$. By \cite[Cor. 2.38]{BS}, $B$ is identified with the $p$-completed PD-envelope of $A$ with respect to $(d)$. Let $\alpha: A\to B$ denote the composite of the canonical map $A\to B$ with the Frobenius $\phi:A\to A$. Then $\alpha$ induces a map of bounded prisms $(A,I)\to (B,(p))$ since $\alpha(I)\subset (p)$.
		Reducing mod $p$, we get a factorization 
		\[\bar\alpha\colon A/p\to A/(p,I)\xra{\bar\phi} A/(p,I^p)\xra{\iota} B/p.\] 
		where the first arrow is the canonical reduction, $\bar{\phi}$ is the Frobenius, and  $\iota$ is induced by the canonical map $A\to B$. Note that $\bar\phi$ is  faithfully flat by  condition (a). We claim that $\iota$ is also faithfully flat as well. Note that  $B/p$ is the PD-envelope of $A/(p)$ with respect to the ideal $(d)$. Since condition (b) implies that $d\in A/(p)$ is nonzero divisior, it follows from   \cite[\href{https://stacks.math.columbia.edu/tag/07HC}{Tag 07HC}]{stacks-project}  that 
		\[
		B/p=A/p\{x\}^{PD}/(x-d),
		\]
		where $A/p\{x\}^{PD}$ is the free divided power ring in one variable over $A/p$. We have an ismomorphism $$A/p\{x\}^{PD}\cong A/p[x_0, x_1, \cdots]/(x_i^p: i\geqslant 0)$$
		where $x_i$ is the image $x^{[p^i]}$. Hence, we get $$B/p\cong A/(p, d^p)[x_1,x_2, \cdots]/(x^p_i:i\geqslant 1)$$
		which is clearly faithfully flat over $A/(p,d^p)$. 
		
		Put $\bar{\beta}=\iota\circ\bar\phi: A/(p,I)\to B/p$.  Then $\bar \beta$ is faithfully flat,  because so are  both $\iota$ and $\bar\phi$. 
		Let  $\bar{\psi}$ be  the reduction of $\psi $ modulo $p$, and $\psi_B =\psi\otimes 1:  R^{PD}(1)\widehat\otimes_{A}B\to R(1)\otimes_{A}B$ be the base change of $\psi$ via $\alpha: A\to B$. Then one has a commutative diagram: 
		\[\xymatrix{
			R/p\{\frac{\xi_1}{d}, \dots, \frac{\xi_n}{d}\} \ar[rr]^{\bar{\psi}}\ar[d] && \widetilde R(1)/(p,I)\widetilde R(1)\ar[d]\\
			(R/p\otimes_{A/(p,I),\bar{\beta}}B/p)\{\frac{\xi_1}{d}, \dots, \frac{\xi_n}{d}\} \ar[rr]^{\bar\psi_B=\bar\psi\otimes 1} && \big(\widetilde R(1)/(p,I)\widetilde R(1)\big){\otimes}_{A/(p,I), \bar{\beta}}B/p,
		}\]
		where the vertical maps are natural inclusions induced by $\bar \beta$, and the bottom map $\bar{\psi}_B$ is reduction of  $\psi_B$.  As the formation of $\psi$ commutes with base change in $(A,I)$, the previous discussion in the case of $(d)=(p)$ implies that $\psi_B$ is injective (hence an isomorphism), it follows that $\bar{\psi}$ is also injective. Let $M$ denote the kernel of $\psi$. 
		Since both the source and the target of $\psi$ are flat over $A/I$, it follows that  $M/pM=\ker(\bar{\psi})=(0)$. As $M$ is separate for the $p$-adic topology, it follows that $M=(0)$. This finishes the proof  of  Proposition~\ref{P:PD-delta}.

	\end{proof}

	Recall that one has a map of cosimplicial objects $\widetilde{R}^{\otimes (\bullet+1)}\to \widetilde R(\bullet)$ (Subsection~\ref{S:cosimplicial object}).
	For  integers $m\geqslant 1$, $1\leqslant i\leqslant n$ and $0\leqslant j\leqslant m$, let  
	$$T_{i,j}:=1\otimes \cdots \otimes1\otimes  T_i\otimes1\otimes \cdots\otimes 1\in \widetilde R^{\otimes (m+1)}$$ with $T_i$ sitting at $j$-th place, and $\xi_{i,j}$ be the image of $T_{i,j}-T_{i,j-1}$ in $\widetilde R(m)$ for $j\geqslant1$. 
In view of Lemma~\ref{L:R-one-structure} and Remark~\ref{R:simplicial product}, if we view $\widetilde R(m)$ as an $\widetilde R$-algebra via ${\delta}^m_m\circ \cdots\circ {\delta}^1_1$ (which corresponds to the projection $\widetilde X(m)\to \widetilde X$ to the $0$-th copy of $\widetilde X$),  one has  an isomorphism  of $\widetilde R$-algebras:
\[
\widetilde R(m)\cong \widetilde R\{\frac{\xi_{i,j}}{I}: 1\leqslant i\leqslant n, 1\leqslant j\leqslant m\}^{ \wedge}
\]
	We put 
	\[
	R^{PD}(m):=\underbrace{R^{PD}(1)\otimes_{R} R^{PD}(1)\otimes_R\cdots \otimes_RR^{PD}(1)}_{m \text{ copies}}.
	\] 
	Similarly as \eqref{E:R-PD-one}, we  write  
	\[R^{PD}(m)= R\{\frac{\xi_{i,j}}{I}: 1\leqslant i\leqslant n, 1\leqslant j\leqslant m\}^{PD, \wedge}\]
	in the coordinates $\xi_{i,j}$.
By Remark~\ref{R:simplicial product},  the isomorphism $\psi: R^{PD}(1)\cong R(1)$ induces isomorphisms 
	\begin{equation}\label{E:isom-psi}
	\psi_m: R^{PD}(m)\xra{\sim}R(m)\cong \underbrace{ R(1)\widehat{\otimes}_{ R}\cdots \widehat{\otimes}_{ R}{R}(1)}_{m \text{ copies}}
	\end{equation}
	for all integers $m\geqslant 1$. We will always use  $\psi_m$ to  identify $R^{PD}(m)$  with $R(m)$. 
	 If $I=(d)$ is principal, then 
	the cosimplicial map $\bar{\delta}^m_{k}:  R(m-1)\to  R(m)$ with $0\leqslant k\leqslant m$ (which is the reduction mod $I$ of \eqref{E:simplicial-projection}) is compatible with the PD-structure and  determined by  
	\begin{equation}\label{C:R-bullet}
	\bar{\delta}_{k}^m\colon \frac{\xi_{i,j}}{d}\mapsto 
	\begin{cases}
		\frac{\xi_{i,j+1}}{d} &\text{if } k<j,\\
		\frac{\xi_{i,j}}{d}+\frac{\xi_{i,j+1}}{d} &\text{if }k=j,\\
		\frac{\xi_{i,j}}{d}&\text{if }k>j,
	\end{cases}
	\end{equation}

	%From now on, we will use the canonical isomorphism in Proposition~\ref{P:PD-delta} to identify $R(1)$ with $R^{PD}(1)$. 

	\subsection{Higgs modules}
		By a Higgs module over $R$, we mean a $p$-completely flat and derived  $p$-complete $R$-module $M$ together with a $R$-linear map 
		\[\theta: M\to M\otimes_R \Omega^1_R\{-1\}\]
		such that the induced map $\theta\wedge\theta: M\to M\otimes \Omega^2_R\{-2\}$ vanishes.  
 Denote by $\rT R:=\Hom_R(\Omega^{1}_R, R)$ the tangent bundle of $R$. Then  $\theta$ induces a map of $R$-modules 
	\[
	\varphi_{\theta}:  \rT R\{1\}=\rT R\otimes_{A/I}I/I^2\to \End_R(M). 
	\]
	The condition $\theta\wedge \theta=0$ is equivalent to saying that two endomorphisms in the image of $\varphi_{\theta}$ commute with each other. 
	Indeed, this claim can be checked after a $(p,I)$-completely faithfully flat base change in  $(A,I)$ so that we may assume that  $I=(d)$ is principal. 
	Then we have    $\Omega^1_R\{-1\}=\oplus_{i=1}^nR \frac{\rd T_i}{d}$, and we can write  $\theta=\sum_{i=1}^n\theta_i \frac{\rd T_i}{d}$ with $\theta_i\in \End_R(M)$.  Then the condition  $\theta\wedge \theta=0$ is equivalent to $\theta_i \theta_j=\theta_j \theta_i$  for all $i,j$. It is clear that  $\{\theta_i, 1\leqslant i\leqslant n\}$ generate the image of $\varphi_{\theta}$, and the claim follows. 
	
	Let $\Sym(\rT R\{1\}):=\bigoplus_{m\geqslant 0}\Sym^m(\rT R\{1\})$ be the symmetric algebra of $\rT R\{1\}$ over $R$. Then $\varphi_{\theta}$ extends to a homomorphism of $R$-algebras:
	\[
	\varphi_{\theta}^\bullet\colon	\Sym( \rT R\{1\})\to \End_R(M).
	\]
	
	\begin{definition}
		A Higgs module $(M,\theta)$ is \emph{topologically quasi-nilpotent} if for each $x\in M$, the submodule $\varphi_{\theta}^\bullet (\Sym^m (\rT\{1\})\cdot x$  of $M$ tends to $0$ as $m\to +\infty$, i.e. for any integer $c>0$ there exists $N>0$ such that $\varphi_{\theta}^\bullet (\Sym^m (\rT\{1\})\cdot x\subset p^c M$ for all $m\geqslant N$.
		Let $\Higgs^{\wedge}(R)$ denote  the category of topologically quasi-nilpotent Higgs modules over $R$.
	\end{definition}
	If $I=(d)$ is principal and $\theta=\sum_{i=1}^n\theta_i \frac{\rd T_i}{d}$ as above, then $(M,\theta)$ is topologically quasi-nilpotent if and only if $\theta^{\mbar}(x)$  tends to $0$ as $|\mbar|:=\sum_{i=1}^nm_i\to +\infty$ for all $x\in M$ and $\mbar\in \NN^n$, where we put
	\[\theta^{\mbar}=\prod_{i=1}^n\theta_i^{m_i}\in \End_R(M),\]
	with  $\theta^{\mbar}=\id$ for $\mbar =(0,\cdots, 0)$.

	\subsection{Stratification and Higgs modules}
	Let $(M,\epsilon)$ be an object of $\Strat( R, R(1))$ (Definition~\ref{D:Obar-stratification}). 
	Note that we have an isomorphism:
	\begin{align*} \Hom_{ R(1)}(M\widehat{\otimes}_R  R(1), M\widehat{\otimes}_R R(1))&\cong \Hom_{R}(M, M\widehat{\otimes}_R R(1))\\
		&\cong \Hom_R(M, M\widehat{\otimes}_R R^{PD}(1))
	\end{align*}
	where the last isomorphism is Proposition~\ref{P:PD-delta}.
	When no confusions arise, we will still  denote  by  $\epsilon$ the image of $\epsilon $ in $\Hom_R(M, M\widehat{\otimes}_R R^{PD}(1))$. 
	Let $$\iota: M\to M\widehat{\otimes}_RR^{PD}(1)$$ 
	be the natural inclusion induced by the structural map $R\to R^{PD}(1)$.
	By the definition of stratification, we have $(\epsilon-\iota)(M)\subset M\widehat\otimes_{R}\calK$. We denote by 
	\[\theta_{\epsilon}: M\to M{\otimes}_R \calK/\calK^{[2]}\cong M{\otimes}_R\Omega_{R}^1\{-1\}\]
	the reduction of $\epsilon-\iota$ modulo $\calK^{[2]}$, where the second isomorphism is \eqref{E:differential-isomorphism}.
	
	\begin{proposition}\label{P:Higgs-stratification}
		The functor  
		$(M,\epsilon)\mapsto (M,\theta_{\epsilon})$ establishes an equivalence of categories 
		\[\Strat(R, {R}(1))\xra{\sim} \Higgs^{\wedge}(R).\]
	\end{proposition}
	\begin{proof}
		We prove first that, for any object $(M, \epsilon)$ of $\Strat( R, R(1))$, the attached object $(M, \theta_{\epsilon})$ is indeed a topologically quasi-nilpotent Higgs module over $R$. Up to base change to a $(p,I)$-completely faithfully flat $\delta$-$A$-algebra, we may assume that $I=(d)$ is principal for a distinguished element $d$.
		We write as usual that 
		\[\theta_{\epsilon}=\sum_{i=1}^n\theta_{\epsilon, i}\,\frac{\rd T_i}{d}.\]
		By Proposition~\ref{P:PD-delta}, we have  
		\[
		M\widehat{\otimes}_RR(1)\cong M\widehat{\otimes}_RR^{PD}(1)=\bigg\{\sum_{\mbar \in \NN^n} x_{\mbar} \bigg(\frac{\xi_{\bullet}}{d}\bigg)^{[\mbar]}: x_{\mbar}\in M, x_{\mbar}\to 0 \text{ as $|\mbar|\to \infty$}\bigg\}
		\]
		where we put \[\bigg(\frac{\xi_{\bullet}}{d}\bigg)^{[\mbar]}=\prod_{i=1}^n\bigg(\frac{\xi_i}{d}\bigg)^{[m_i]}.\]
		For $x\in M$, we write 
		\[
		\epsilon(x)=\sum_{\mbar\in \NN^n}\Theta_{\mbar}(x) \bigg(\frac{\xi_{\bullet}}{d}\bigg)^{[\mbar]}\in M\widehat{\otimes}_RR^{PD}(1)
		\]
		for some $\Theta_{\mbar}\in \End_R(M)$ with $\Theta_{\mbar}=\id_M$ if $\mbar=(0,\dots, 0)$. By the definition of $\theta_{\epsilon}$, it is clear that $\theta_{\epsilon,i}=\Theta_{\underline e_i}$ with $\underline{e}_i\in \NN^n$ the element with the $i$-th component equal to $1$ and others components equal to $0$.
		To see that $(M, \theta_{\epsilon})$ is an object of $\Higgs^{\wedge}(R)$, it suffices to show that 
		\begin{equation}\label{E:commutativity-theta}
			\theta_{\epsilon,i}\theta_{\epsilon,j}=\theta_{\epsilon, j}\theta_{\epsilon,i},\quad 
			\Theta_{\mbar}=\prod_{i=1}^n\theta_{\epsilon,i}^{m_i}=\theta_{\epsilon}^{\mbar}.\end{equation}
		According to \eqref{E:isom-psi} and \eqref{C:R-bullet}, we have  $ R(2)\cong R^{PD}(2)=R\{\frac{\xi_{i,j}}{d}: 1\leqslant i\leqslant n, 1\leqslant j\leqslant 2\}$
		and the maps $\bar\delta^{2}_{j}:  R(1)\to  R(2)$ with $j=0,1,2$ are given by 
		\begin{align*}
			&\bar\delta^2_0(\frac{\xi_i}{d})=\frac{\xi_{i,2}}{d},
			&\bar\delta^2_{1}(\frac{\xi_i}{d})&=\frac{\xi_{i,1}}{d}+\frac{\xi_{i,2}}{d},
			&\bar\delta^2_{2}(\frac{\xi_i}{d})&=\frac{\xi_{i,1}}{d}.
		\end{align*}
		for all $1\leqslant i\leqslant n$. Then for all $x\in M$, we have  
		\begin{align}\label{E:calcul-cocycle}
			\bar \delta^{2,*}_{1}(\epsilon(x))&=\sum_{\mbar \in \NN^n}\Theta_{\mbar}(x)\bigg(\frac{\xi_{\bullet, 1}}{d}+\frac{\xi_{\bullet,2}}{d}\bigg)^{[\mbar]}\\
			&=\sum_{\mbar_1\in \NN^n}\sum_{\mbar_2\in \NN^n}\Theta_{\mbar_1+\mbar_2}(x) \bigg(\frac{\xi_{\bullet, 1}}{d}\bigg)^{[\mbar_1]}\bigg(\frac{\xi_{\bullet, 2}}{d}\bigg)^{[\mbar_2]},\nonumber \\
			\bar \delta^{2,*}_{2}(\epsilon)\big(\bar{\delta}_{0}^{2,*}(\epsilon(x))\big)&=\sum_{\mbar_1\in \NN^n}\sum_{\mbar_2\in \NN^n} \Theta_{\mbar_1}(\Theta_{\mbar_2}(x)) \bigg(\frac{\xi_{\bullet,1}}{d}\bigg)^{[\mbar_1]}\bigg(\frac{\xi_{\bullet,2}}{d}\bigg)^{[\mbar_2]}.\nonumber 
		\end{align}
		Then the cocycle condition $\bar \delta^{2,*}_1(\epsilon)=\bar\delta^{2,*}_2(\epsilon)\circ \bar \delta_0^{2,*}(\epsilon)$ is equivalent to saying that 
		\[\Theta_{\mbar_1+\mbar_2}=\Theta_{\mbar_1}\circ \Theta_{\mbar_2}, \]
		for all $\mbar_1,\mbar_2\in \NN^n$, from which \eqref{E:commutativity-theta} follows immediately.
		
		To finish the proof of the Proposition, it suffices to construct a functor quasi-inverse to $(M,\epsilon )\mapsto (M,\theta_{\epsilon})$.   Let $(M,\theta)$ be an object of $\Higgs^{\wedge}(R)$. We need to construct a stratification $\epsilon$ on $M$ over $ R(1)$. By the usual descent argument, we may reduce the problem to the case when $I=(d)$ is principal.
		For any $x\in M$, we put 
		\begin{equation}\label{E:formula-esp}
			\epsilon(x):=\sum_{\mbar\in \NN^n}\theta^{\mbar}(x)\bigg(\frac{\xi_{\bullet}}{d}\bigg)^{[\mbar]}\in M\widehat{\otimes}_RR^{PD}(1)\cong M\widehat{\otimes}_R R(1),\end{equation}
		which is well defined by the topological quasi-nilpotence of $(M,\theta)$.
		This defines an element 
		\[\epsilon\in  \Hom_R(M,M\widehat\otimes_RR(1))\cong \End_{R(1)}\big(M\widehat\otimes_RR(1)\big).\]
		We prove first  that $\epsilon$ is an isomorphism of $M\widehat\otimes_R R(1)$. 
		Let  $\epsilon'$ be the endomorphism of $M\widehat\otimes_RR(1) $ defined by 
		\[\epsilon'(x)=\sum_{\mbar\in \NN^n}(-1)^{|\mbar|}\theta^{\mbar}(x)\bigg(\frac{\xi_{\bullet}}{d}\bigg)^{[\mbar]}\quad \forall x\in M.\]
		We claim that $\epsilon'$ is the inverse of $\epsilon.$. 
		Indeed, we have 
		\begin{align}\label{E:epsilon-formula}
			\epsilon'(\epsilon(x))&=\epsilon'\bigg(\sum_{\mbar \in \NN^n}\theta^{\mbar}(x)\bigg(\frac{\xi_{\bullet}}{d}\bigg)^{[\mbar]} \bigg)\\
			&=\sum_{\mbar\in \NN^n}\sum_{\mbar'\in \NN^n}(-1)^{|\mbar'|}\theta^{\mbar'}(\theta^{\mbar}(x))\bigg(\frac{\xi_{\bullet}}{d}\bigg)^{[\mbar']}\bigg(\frac{\xi_{\bullet}}{d}\bigg)^{[\mbar]}\nonumber \\
			&=\sum_{\underline{s}\in \NN^n}\theta^{\underline{s}}(x)\sum_{\mbar'\in \NN^n, \mbar' \leqslant \underline{s}} (-1)^{|\mbar'|} \bigg(\frac{\xi_{\bullet}}{d}\bigg)^{[\mbar']}\bigg(\frac{\xi_{\bullet}}{d}\bigg)^{[\underline{s}-\mbar']}.\nonumber
		\end{align}
		Here, the notation $\mbar'\leqslant \underline{s}$  in the last equality means $m'_i\leqslant s_i$ for all $1\leqslant i\leqslant n$. Note that if $s_i>0$ for some $i$, then  we have 
		\[	\sum_{m'_i=0}^{s_i} (-1)^{m'_i} \bigg(\frac{\xi_{i}}{d}\bigg)^{[m'_i]}\bigg(\frac{\xi_{i}}{d}\bigg)^{[s_i-m'_i]}=\sum_{m'_i=0}^{s_i}\bigg(\frac{\xi_i}{d}-\frac{\xi_i}{d}\bigg)^{s_i}=0.
		\]
		Therefore, the contribution of the terms with $|\underline s|>0$ to \eqref{E:epsilon-formula} is zero, and we deduce that $\epsilon'(\epsilon(x))=x$, i.e. $\epsilon'\circ\epsilon=\id$. A similar computation shows   as well that $\epsilon\circ\epsilon'=\id$. 
		It remains to check that $\epsilon $ verifies the cocycle condition $\bar \delta^{2,*}_1(\epsilon)=\bar\delta^{2,*}_2(\epsilon)\circ \bar \delta_0^{2,*}(\epsilon)$.  By the computation \eqref{E:calcul-cocycle}, this   follows immediately from  the obvious fact that  $\theta^{\mbar_1+\mbar_2}=\theta^{\mbar_1}\circ\theta^{\mbar_2}$. This construction gives a functor $ \Higgs^{\wedge}(R)\to \Strat(R, R(1))$, which is easily seen to be  a quasi-inverse to $(M, \epsilon)\mapsto (M, \theta_{\epsilon})$.
	\end{proof}

	Combining now Propositions~\ref{P:crystal-stratification} and \ref{P:Higgs-stratification}, we obtain the following
	\begin{theorem}\label{T:crystals-Higgs}
		Suppose that we are in Situation~\ref{A:local-coordinate}. Then we have a sequence of equivalence of categories
		\[\CR((X/A)_{\prism}, \overline{\cO}_{\prism})\xra{\overline{\ev}_{\widetilde X}}\Strat(R,R(1)) \xra{\mathrm{Prop.}\; \ref{P:Higgs-stratification}}\Higgs^{\wedge}(R).\] 
	\end{theorem}
	
	\begin{remark}
		(1)	It is natural to expect that there still exists an equivalence of categories between $\CR((X/A)_{\prism}, \overline{\cO}_{\prism})$ and $\Higgs^{\wedge}(R)$ under the more general assumption~\ref{A:local-lift}.

		(2) Gros--Le Strum--Quir\'os \cite[Cor. 6.6]{GSQ} and Morrow--Tsuji  \cite[Thm.~3.2]{MS} give  similar descriptions of ${\cO}_{\prism}$-crystals when $(A,I)$ is a bounded prism over $(\ZZ_p[[q-1]], ([p]_q))$ with $[p]_q=\frac{q^p-1}{q-1}$. In these special cases, Theorem~\ref{T:crystals-Higgs} is compatible with their results after modulo $I$.
	\end{remark}

	We now describe  the cohomology of an $\overline{\cO}_{{\prism}}$-crystal in terms of  the de Rham complex of its associated Higgs module. %Let $(M,\theta)$ be an object of $\Higgs^{\wedge}(R)$. We denote by $\DR(M,\theta)$ the de Rham complex induced by $\theta$:
	%\[M\xra{\theta}M\otimes_R\Omega^1_R\{-1\}\xra{\theta}\cdots \xra{\theta} M\otimes_{R}\Omega^n_R\{-n\}.\] 
	
	\begin{theorem}\label{T:cohomology-Higgs}
		Suppose that we are in Situation~\ref{A:local-coordinate}.
		Let $\calE$ be an $\overline{\cO}_{\prism}$-crystal, and $(\calE_{\widetilde X},\theta)$ be its associated object of $\Higgs^{\wedge}(R)$ via Theorem~\ref{T:crystals-Higgs}. Then $R\nu_{X/A,*}(\calE)$ is computed by the de Rham complex of $(\calE_{\tilde X}, \theta)$:
		\[ \DR(\calE_{\widetilde X},\theta)\colon =(\calE_{\widetilde X}\xra{\theta}\calE_{\widetilde X}\otimes_R\Omega^1_R\{-1\}\xra{\theta}\cdots \xra{\theta} \calE_{\widetilde X}\otimes_{R}\Omega^n_R\{-n\})
		\]
		% In particular, if $\overline{\calE}$ is locally free of finite type over  $\overline{\cO}_{\prism}$, then $R\nu_{X/A,*}(\calE)$ is a perfect complex of $\cO_X$-modules.
		
	\end{theorem}

\if false 

Before proving this Theorem, we give an application to the base change properties of cohomology of $\cO_{\prism}$-crystals.

	\begin{theorem}[base change]\label{C: base change}
	Let $X$ be a smooth $p$-adic formal scheme over $\Spf(A/I)$.
	Let $\calF$ be an $\cO_{\prism}$-crystal on $(X/A)_{\prism}$. 
	Let $(A,I)\to (A',I')$ be a morphism of bounded prisms of finite tor-dimension,  $X'=X\times_{\Spf(A/I)}\Spf(A'/I')$ and $f:X'\to X$  be the natural projection.
	 Consider the following commutative diagram of morphisms of topoi:
	\[
	\xymatrix{
		(X'/A')^{\sim}_{\prism}\ar[r]^{f_{\prism}}\ar[d]_{\nu'} & (X/A)^{\sim}_{\prism}\ar[d]^{\nu}\\
		X_{\Et}^{'\sim}\ar[r]^{f_{\Et}} &X^{\sim}_{\Et}}
	\]
	Let $\calF'=f_{\prism}^*\calF$ be the pullback of $\calF$ in the sense of $\cO_{\prism}$-crystals. Then there exists a canonical base change  isomorphism 
	\begin{equation*}f_{\Et}^{-1}R\nu_{*}(\calF)\otimes^L_AA')^{\wedge}\xra{\sim} R\nu'_{*}( \calF') 
	\end{equation*}
	is  an isomorphism.
\end{theorem}
\begin{proof}
	We construct first a canonical  morphism $(f_{\Et}^{-1}R\nu_{*}(\calF)\otimes^L_AA')^{\wedge}\to R\nu'_{*}( \calF') $. By adjunction, it suffices to product a morphism $R\nu_*$ of sheaves on $X_{\Et}$
	
	The problem is clearly local for the \'etale topology of $X$. Up to \'etale localization, we may impose thus Assumption~\ref{A:local-lift}. 
	In this case, the statement  follows immediately from Prop.~\ref{P:CA-complex} and the fact that the formation of \v{C}ech--Alexander complex \eqref{E:CA-complex} commutes with the base change  $A\to A'$.
\end{proof}

\begin{remark}
	After establishing Theorem~\ref{T:cohomology-Higgs}, we will see that Corollary~\ref{C:weak base change} holds without the assumption that   $A\to A'$ is of finite tor-dimension. 
\end{remark}

\fi

We  prove Theorem~\ref{T:cohomology-Higgs} by following the strategy of \cite{BD}.
	%First, we define a cosimplicial object $\Omega^j_{R(\bullet)}$ for each integer $j\geqslant 0$. 
%	By Corollary~\ref{C:weak base change}, up to replacing $A$ by a $(p,I)$-completely faithfully flat $\delta$-$A$-algebra, we may assume that $I=(d)$ is principal
For an integer $m\geqslant 0$,   we have an    isomorphism  \eqref{E:isom-psi}   
	$$R(m)=\widetilde R(m)/I\widetilde R(m)\cong R^{PD}(m)=R\{\frac{\xi_{i,j}}{I}: 1\leqslant i\leqslant n, 1\leqslant j\leqslant m\}^{PD, \wedge},$$
	 where   $R^{PD}(m)$ is a certain twisted divided power polynomial ring over $R$ in $mn$-variables. The kernel  $\calK(m)$ of the canonical surjection  $R(m)\to R$  has a divided power structure. 
	Let  $\Omega^1_{R(m)/R}$ denote the module of continuous differential $1$-forms relative to  $R$ compatible with the divided power structure, i.e. $\rd x^{[r]}=\rd x^{[r-1]}$ 
	for all integer $r\geqslant 1$  and $x\in \calK(m)$.   If $I=(d)$ is principal, this  is a locally free $R(m)$-module with basis $\{\frac{\rd\xi_{k,l}}{d}:1\leqslant k\leqslant n, 1\leqslant l\leqslant m\}$. Here, $\frac{\rd \xi_{k,l}}{d}$ should be rather understood as the differential of the variable $\frac{\xi_{k,l}}{d}\in R(m)$, because the natural image of $\xi_{k,l}=T_{k,l}-T_{k,l-1}\in \widetilde R(m)$ in $R(m)$ is zero. 
	For an integer $j\geqslant 1$, we put $\Omega^j_{R(m)/R}=\wedge^j \Omega^1_{R(m)/R}$, where the wedge product is over $R(m)$. 
	
	We put 
	\[\Omega^{1}_{R(m)}\colon =\Omega^1_{R}\{-1\}\otimes_RR(m)\bigoplus \Omega^1_{R(m)/R}\]
	which is noncanonically isomorphic to the module of  continuous differential $1$-forms of $R(m)$ with respect to $A/I$ compatible with the divided power structure. For each map of simplices $f:[m]\to [m']$, we will construct  a map 
	\[f_*\colon \Omega^1_{R(m)}\to \Omega^1_{R(m')} %\quad \frac{\rd T_{k,l}}{d}\mapsto \frac{\rd T_{k,f(l)}}{d}.
	\]
	such that  $(g\circ f)_*=g_*\circ f_*$ for any two composable maps of simplices. We assume first that $I=(d)$ is principal.  Then $\Omega^1_{R(m)}$ is a free $R(m)$-module with a basis given by 
	 $$\{\frac{\rd T_k}{d}=\rd T_k\otimes d^{-1}: 1\leqslant k\leqslant n\}\cup \{\frac{\rd \xi_{k, l}}{d}: 1\leqslant k\leqslant n, 1\leqslant l\leqslant m\}$$
	By the discussion after the proof of Proposition~\ref{P:PD-delta},   $\xi_{k,l}$ should be understood as the image of $T_{k,l}-T_{k,l-1}\in \widetilde R^{\otimes(n+1)}$ and $T_k$ is identified the image of $T_{k,0}$.
If we sets
	\[
	\frac{\rd T_{k,0}}{d}:=\frac{\rd T_{k}}{d}, \quad \frac{\rd T_{k,l}}{d}:=\frac{\rd T_k}{d}+\sum_{s=1}^{l}\frac{\rd \xi_{k,s}}{d} \quad \quad \text{for }  1\leqslant l \leqslant m,
	\] 
	then $\{\frac{\rd T_{k,l}}{d}: 1\leqslant k \leqslant n, 0\leqslant l\leqslant m\}$ 
	gives  another basis of $\Omega^1_{R(m)}$ over $R(m)$. 
	For each map of simplices $f:[m]\to [m']$, we define
		\[f_*\colon \Omega^1_{R(m)}\to \Omega^1_{R(m')},\quad \frac{\rd T_{k,l}}{d}\mapsto \frac{\rd T_{k,f(l)}}{d}.
	\]
Such a definition is independent of the choice of the generator $d$ of $I$. If $I$ is not necessarily principal, we take a $(p,I)$-completely faithfully flat $\delta$-$A$-algebra $A'$ such that $I'=IA$ is principal, and make the construction of $f_*$ after base change to $A'$. An usual descent argument with  Proposition~\ref{P:flat-descent} allows us to descend the morphism $f_*$ to  $A$. It is also direct to see that $(g\circ f)_*=g_*\circ f_*$ for two composable maps of simplices.  
	This defines a cosimplicial  module $\Omega^1_{R(\bullet)}$ over the cosimplicial ring $R(\bullet)$. Taking wedges, we get a cosimplicial $R(\bullet)$-module $\Omega^j_{R(\bullet)}$ for each integer $j\geqslant 1$.
	
	\begin{lemma}\label{L:trivial-simplicial}
		For each integer $j\geqslant1$, the cosimplicial object $\Omega^j_{R(\bullet)}$ is homotopic to zero.
	\end{lemma}
	\begin{proof}
		For $j=1$, the  proof is the same as \cite{BD}*{Lemma 2.15}. Indeed, the cosimplicial object $\Omega^1_{R(\bullet)}$ is a direct sum of $n$ copies of cosimplicial objects in  \cite{BD}*{Example 2.16}. Therefore, it is homotopic to zero. The general case follows since being homotopic to zero is stable under the application of  the functor $\wedge^j$ termwise.	
	\end{proof}
	Note that 
	\begin{equation}\label{E:projection-Omega}
		\Omega^j_{R(m)}=\bigoplus_{k=0}^j \Omega^k_R\{-k\}\otimes_R\Omega^{j-k}_{R(m)/R}.
		\end{equation}
	Let $\rd_R: \Omega^j_{R(m)}\to \Omega^{j+1}_{R(m)}$ be the $R$-linear map given by 
	\[\rd_R(\sum_{k=0}^j\omega_k\otimes \eta_{j-k})=\sum_{k=0}^j(-1)^{k}\omega_k\otimes \rd\eta_{j-k}\]
	for $\omega_k\in \Omega^k_{R}\{-k\}$ and $\eta_{j-k}\in \Omega_{R(i)/R}^{j-k}$.
	In other words, $(\Omega^{\bullet}_{R(m)}, \rd_R)$ is the simple complex attached to the tensor product of  
	\[R\xra{0}\Omega^1_{R}\{-1\}\xra{0}\cdots\xra{0} \Omega^{n}_{R}\{-n\}\] with the relative de Rham complex $(\Omega^{\bullet}_{R(m)/R}, \rd)$. One checks easily that $\rd_R: \Omega^j_{R(\bullet)}\to \Omega^{j+1}_{R(\bullet)}$ is a map of cosimplicial objects, i.e. for any map of simplices $f: [m]\to [m']$, $\rd_R$ commutes with $f_*\colon \Omega^j_{R(m)}\to \Omega^{j}_{R(m')} $.
	
	Let $\widetilde X(\bullet)$ be  the simiplicial object  in $(X/A)_{\prism}$ defined in \S\ref{S:cosimplicial object}. For an integer $i\geqslant 0$, let $\calE_{\widetilde X(i)}$ be the evaluation of $\calE$ at $\widetilde X(i)$. As $\calE$ is  an $\overline{\cO}_{\prism}$-crystal, we have an isomorphism
	\begin{equation}\label{E:identify E}c_{\pr_0}(\calE)\colon \calE_{\widetilde X(i)}\xra{\sim}\calE_{\widetilde X}\otimes_R R(i)\end{equation}
	induced by the $0$-th projection $\pr_0: \widetilde X(i)\to \widetilde X$. In the sequels, we will always use $c_{\pr_0}(\calE)$ to identity these two modules. 
	For  integers $i,j\geqslant 0$, we put 
	\[\scrC^{i,j}=\calE_{\widetilde X(i)}\otimes_{ R(i)}\Omega^j_{ R(i)}\cong \calE_{\widetilde X}\otimes_R \Omega^{j}_{R(i)}.\]
	Then for each integer $j\geqslant 0$, $\scrC^{\bullet, j}=\calE_{\widetilde X(\bullet)}\otimes_{R(\bullet)}\Omega^j_{R(\bullet)}$has a natural structure of  cosimplicial $R(\bullet)$-modules. Let  $(\scrC^{\bullet,j}, d_1^{\bullet, j})$ be the chain complex associated to the cosimplicial object $\scrC^{\bullet, j}$, i.e. if $\delta^{i+1}_{k,*}$ is  the unique injective order preserving  map $\delta^{i+1}_{k}: [i]:=\{0,1,\dots, i\}\to [i+1]$ that skips $k$,  we have 
	\begin{equation}\label{E:simplicial-formula}
		d_1^{i,j}=\sum_{k=0}^{i+1} (-1)^k\delta^{i+1}_{k,*}: \scrC^{i,j}\to \scrC^{i+1,j}.\end{equation}
	We also define  
	$d_2^{i,j}\colon \scrC^{i,j}\to \scrC^{i,j+1}$
	by 
	\begin{equation*}\label{E:formula-d2}
	d_2^{i,j}(x\otimes \omega)=\theta(x)\wedge \omega+x\otimes \rd_R\omega
	\end{equation*}
	for  $x\in \calE_{\widetilde X}$ and $\omega\in \Omega_{R(i)}^1$. It is easy to see that, for each $i\geqslant0$,   $(\scrC^{i,\bullet}, d_2^{i,\bullet})$ is the simple complex associated to the tensor product of $\DR(\calE_{\widetilde X(i)}, \theta)$ with the relative de Rham complex $\Omega^{\bullet}_{R(i)/R}$ (see \ref{S:sign-covention} for the convention).
	\begin{lemma}\label{L:double complex}
		For each integer $j\geqslant 0$, $d_{2}^{\bullet,j}\colon \scrC^{\bullet,j}\to \scrC^{\bullet, j+1}$ is a morphism of cosimplicial objects. In particular, $(\scrC^{\bullet, \bullet}, d_1^{\bullet,\bullet}, d_2^{\bullet, \bullet})$ is a na\"ive double complex.
	\end{lemma}
	\begin{proof}
		For integers  $i,k\geqslant0$  with $ k\leqslant i$, let $\sigma^{i+1}_k:[i+1]\to [i]$ be the surjective order preserving map of simplices with $(\sigma_k^{i+1})^{-1}(k)=\{k,k+1\}$.
		It suffices to prove that $\sigma^{i}_{k,*}d_2^{i+1,j}=d_{2}^{i,j}\sigma^{i}_{k,*}$ and   $d_{2}^{i+1,j}\delta^{i+1}_{k,*}=\delta^{i+1}_{k,*} d_2^{i,j}$. 
		
		For $y\otimes\eta \in \calE_{\widetilde X}\otimes_R\Omega^j_{R(i+1)}=\scrC^{i+1,j}$, we have 
		\[
		\sigma^{i}_{k,*}(y\otimes \eta)=y\otimes \sigma^{i}_{k,*}(\eta).
		\]
		Then the commutation of $d_2$ with $\sigma^{i}_{k,*}$ follows from the fact  that $\rd_R: \Omega^j_{R(\bullet)}\to \Omega^{j+1}_{R(\bullet)}$ is a morphism  of cosimplicial objects.

		It remains to check $d_{2}^{i+1,j}\delta^{i+1}_{k,*}=\delta^{i+1}_{k,*} d_2^{i,j}$.  As usual, it suffices to check this equality after  base change to a $(p,I)$-completely faithfully flat $\delta$-$A$-algebra $A'$ with $IA'$ principal. Therefore, up to changing the notation, we may assume that $I=(d)$ is principal.
		For $x\otimes \omega\in \calE_{\widetilde X}\otimes_R\Omega^j_{R(i)}$, which is identified with $\calE_{\widetilde X(i)}\otimes_{R(i)}\Omega^j_{R(i)}$ via \eqref{E:identify E}, we have 
		\[\delta_{k,*}^{i+1}(x\otimes \omega)=\begin{cases}
			x\otimes \delta^{i+1}_{k,*}(\omega), & \text{if }k\geqslant1;\\
			\pr_{0,1,*}(\epsilon)(x)\otimes \delta^{i+1}_{0,*}(\omega), &\text{if }k=0.
		\end{cases}\]
		Here, $\pr_{0,1,*}: R(1)\to R(i+1)$ is the map induced by the natural inclusion of simplices $[1]\to [i+1]$.
		 By the formula \eqref{E:formula-esp}, we get 
		\[\pr_{0,1,*}(\epsilon)(x)=\sum_{\underline m\in \NN^n}\theta^{\underline m}(x)\otimes \bigg(\frac{\xi_{\bullet,1}}{d}\bigg)^{[\underline m]}
		\]
		Assume first that $k\geqslant1$. Then $d_2^{i+1,j}\delta^{i+1}_{k,*}=\delta^{i+1}_{k,*}d_2^{i,j}$ amounts to the commutation of $\rd_R$ with $\delta_{k,*}^{i+1}$. Consider now the case $k=0$. Then for $x\otimes \omega\in \calE_{\widetilde X}\otimes_R\Omega^j_{R(i)}$, we have 
		\begin{align*}d_2^{i+1,j} (\delta_{0,*}^{i+1}(x\otimes \omega))&=\sum_{\underline m\in \NN^n}d_2^{i+1,j}\bigg(\theta^{\underline m}(x)\otimes\bigg(\frac{\xi_{\bullet,1}}{d}\bigg)^{[\underline m]} \delta^{i+1}_{0,*}(\omega)\bigg)\\
			&=\sum_{\underline m\in \NN^n}\sum_{l=1}^n \theta_l(\theta^{\underline m}(x))\otimes \bigg(\frac{\xi_{\bullet,1}}{d}\bigg)^{[\underline m]}\frac{\rd T_l}{d}\wedge  \delta^{i+1}_{0,*}(\omega)\\
			&+\sum_{\underline m\in \NN^n, |\underline m |\geqslant 1}\sum_{l=1}^n\theta^{\underline m}(x)\otimes \bigg(\frac{\xi_{\bullet,1}}{d}\bigg)^{[\underline m-e_l]} \frac{\rd\xi_{l,1}}{d}\wedge \delta_{0,*}^{i+1}(\omega)+ \pr_{0,1,*}(\epsilon)(x)\otimes \rd_R(\delta_{0,*}^{i+1}\omega)\\
			&=\sum_{\underline m\in \NN^n}\sum_{l=1}^n \theta_l(\theta^{\underline m}(x))\otimes \bigg(\frac{\xi_{\bullet,1}}{d}\bigg)^{[\underline m]}\bigg(\frac{\rd T_l}{d}+\frac{\rd \xi_{l,1}}{d}\bigg)\wedge  \delta^{i+1}_{0,*}(\omega)+ \pr_{0,1,*}(\epsilon)(x)\otimes \rd_R(\delta_{0,*}^{i+1}\omega).
		\end{align*}
		Here, $e_l\in \NN^n$ is the element with $1$ at $l$-th component and $0$ at other components.
		On the other hand, we have 
		\begin{align*}
			\delta^{i+1}_{0,*}(d_2^{i,j}(x\otimes \omega))&=\delta^{i+1}_{0,*}\bigg(\sum_{l=1}^n\theta_l(x)\otimes\frac{\rd T_l}{d}\wedge \omega+ x\otimes \rd_R(\omega)\bigg)\\
			&=\sum_{\underline m\in \NN^n}\sum_{l=1}^n\theta^{\underline m}(\theta_l(x))\otimes \delta^{i+1}_{0,*}(\frac{\rd T_l}{d})\wedge \delta^{i+1}_{0,*}(\omega)+ \pr_{0,1,*}(\epsilon)(x)\otimes \delta^{i+1}_{0,*}(\rd_R\omega).
		\end{align*}
		Now $(\delta_{0,*}^{i+1}(x\otimes \omega))=\delta^{i+1}_{0,*}(d_2^{i,j}(x\otimes \omega))$ follows from the fact that $\delta^{i+1}_{0,*}(\frac{\rd T_{l}}{d})=\frac{\rd T_{l}}{d}+\frac{\rd \xi_{l,1}}{d}$ and that $\delta^{i+1}_{0,*}$ commutes with $\rd_R$.
	\end{proof}

	\begin{lemma}\label{L:poincare-lemma}
		For   an integer $i\geqslant 0$,  any morphism of simplices  $\delta: [0]\to [i]$ induces a  quasi-isomorphism of  complexes 
		$$\DR(\calE_{\widetilde X}, \theta)=(\scrC^{0, \bullet}, d_2^{0,\bullet})\xra{\sim}(\scrC^{i,\bullet}, d_2^{i,\bullet}).$$
		
	\end{lemma}
	\begin{proof}
			Note first that   $(\scrC^{i,\bullet}, d_2^{i,\bullet})$ is the simple complex attached to the tensor product of $\DR(\calE_{\tilde X}, \theta)$ with the relative de Rham complex $(\Omega^\bullet_{R(i)/R},d)$; in particular, one has $\DR(\calE_{\widetilde X}, \theta)=(\scrC^{0, \bullet}, d_2^{0,\bullet})$.
			 By  Lemma~\ref{L:double complex}, $d_2^{\bullet, j}$ is a morphism of cosimplicial $R(\bullet)$-modules for each fixed  $j\geqslant 0$. If we let $j$ vary,  any morphism of simplices $\delta:[0]\to [i]$ induces a map of complexes
		 of complexes   
		 $\delta_{*}: (\scrC^{0, \bullet}, d_2^{0,\bullet})\to (\scrC^{i, \bullet}, d_2^{i,\bullet})$. 
		 We have to see that this is a quasi-isomorphism. 

		First, we claim that the natural map of complexes $R\xra{\sim} \Omega^{\bullet}_{R(i)/R}$ is quasi-isomorphism. Indeed, this can be checked after a $(p,I)$-completely faithfully flat base change in $(A,I)$. Up to making such a base change, we may assume that $I=(d)$ is principal. Then $R(i)$ is a free divided power polynomial ring over $R$ in $ni$-variables, from which the claim follows. As a consequence,  the natural  projections  $\Omega^j_{R(i)}\to\Omega^j_{R}\{-j\}$  in the decomposition \eqref{E:projection-Omega}  for all $j$  give a  quasi-isomorphism
		\[
		(\scrC^{i,\bullet}, d_2^{i,\bullet})\xra{\sim} 	\DR(\calE_{\widetilde X}, \theta).
		\]
	Since the composition of $\delta_{*}$ with the above contraction map is the identity on $\DR(\calE_{\widetilde X}, \theta)$, the Lemma follows immediately.

	\end{proof}

	\begin{proof}[End of the proof of Theorem~\ref{T:cohomology-Higgs}]
		Note that $\scrC^{\bullet, 0}$ is by definition the \v{C}ech--Alexander complex $\CA(\widetilde X(\bullet), \calE)$ (cf. \eqref{E:CA-complex}), and $\DR(\calE_{\widetilde X}, \theta)=(\scrC^{0, \bullet}, d_2^{0,\bullet})$.   By Proposition.~\ref{P:CA-complex},  $R\nu_{X/A,*}(\calE)$ is computed by $\scrC^{\bullet, 0}$.  To finish the proof, it suffices to show that both $\scrC^{\bullet, 0}$ and $\scrC^{0,\bullet}$ are both quasi-isomorphic to $\underline{s}(\scrC^{\bullet, \bullet})$, the simple complex associated to $\scrC^{\bullet,\bullet }$. Recall that the map of complexes 
		$d_{1}^{i,\bullet}\colon \scrC^{i,\bullet}\to \scrC^{i+1,\bullet}$
		is given  by $\sum_{k=0}^{i+1}(-1)^k\delta^{i+1}_{k,*}$. By Lemma~\ref{L:poincare-lemma}, $\scrC^{i,\bullet}$ and $\scrC^{i+1,\bullet}$ are both  quasi-isomorphic to $\DR(\calE_{\tilde X},\theta)$ and each $\delta^{i+1}_{k,*}$ corresponds to the identity map  of $\DR(\calE_{\tilde X},\theta)$.
		Therefore, via these quasi-isomorphisms, 
		$d_1^{i,\bullet}$ corresponds to  the identity map of $\DR(\calE_{\tilde X}, 
		\theta)$ if $i$ is odd, and  to  the zero map if $i$ is even. It follows that $\underline{s}(\scrC^{\bullet, \bullet})$ is quasi-isomorphic to 
		\[\underline {s}\big(\DR(\calE_{\tilde X}, \theta)\xra{0} \DR(\calE_{\tilde X}, \theta)\xra{\id}\DR(\calE_{\tilde X}, \theta)\xra{0}\cdots \big)\cong \DR(\calE_{\tilde X}, \theta)=\scrC^{0,\bullet}.\]
		
		On the other hand, by Lemma~\ref{L:trivial-simplicial}, the complex $\scrC^{\bullet, j}=\calE_{\widetilde X(\bullet)}\otimes_{R(\bullet)} \Omega^{j}_{R(\bullet)}$ is homotopic to zero for all integers $j\geqslant 1$ since being homotopic to zero is stable under  tensoring with another cosimplicial module. Hence, the natural map
		\[\scrC^{\bullet, 0}\xra{\sim} \underline{s}(\scrC^{\bullet, \bullet}) \]
		is a quasi-isomorphism. This finishes the proof of Theorem~\ref{T:cohomology-Higgs}.

	\end{proof}
	
	\begin{remark}
		It is clear that the trivial $\overline{\cO}_{\prism}$-crystal $\calE=\overline{\cO}_{\prism}$ corresponds to the trivial Higgs module $(\cO_X, 0)$. Thus 
		Theorem~\ref{T:cohomology-Higgs} implies that 
		\[
		R\nu_{X/A,*}(\overline \cO_{\prism})=\bigoplus_{i=0}^{n}\Omega^1_R\{-i\}.
		\] 
		From this, it is easy to deduce Bhatt--Scholze's Hodge--Tate comparison theorem \cite[Theorem~4.11]{BS} for general smooth $p$-adic formal scheme over $\Spf(A/I)$. %Theorem~\ref{T:cohomology-Higgs} 
		%can be thus viewed as a local generalization of Hodge--Tate comparison theorem for general $\overline{\cO}_{\prism}$-crystals.
	\end{remark}
	
	%We can now finish the proof of Theorem~\ref{T:Obar-crystal}. 
	\subsection{}\label{S:proof of main theorem} \textbf{End of the proof of  ~\ref{T:Obar-crystal}:}
 First, we note that the perfectness of  $R\nu_{X/A,*}(\calE)$ can be checked after base change to a $(p,I)$-completely faithfully flat  $\delta$-$A$-algebra $A'$ with $IA'$ principal. Indeed,  as $R\nu_{X/A,*}(\calE)$ is derived $p$-complete,  by \cite[\href{https://stacks.math.columbia.edu/tag/09AW}{Tag 09AW}]{stacks-project}, it suffices to show that $R\nu_{X/A, *}(\calE)\otimes_{A/I}^LA/(I, p^n)$ is a perfect $\cO_{X}\otimes_A{A/(I,p^n)}$-module for each integer $n\geqslant 1$. Then according to \cite[\href{https://stacks.math.columbia.edu/tag/068T}{Tag 068T}]{stacks-project}, the perfectness of $R\nu_{X/A, *}(\calE)\otimes_{A/I}^LA/(I,p^n)$ can be checked after the faithfully flat base change $A/(I, p^n)\to A'/(I,p^n)$. Therefore, we are reducing to showing  that $(R\nu_{X/A,*}(\calE)\otimes_A^LA')^{\wedge}$ is a perfect $\cO_{X'}$-module, where $X'=X\times_{\Spf(A/I)}\Spf(A'/IA')$. 
	By Corollary~\ref{C:weak base change}, one has a canonical isomorphism  $(R\nu_{X/A,*}(\calE)\otimes_A^LA')^{\wedge}\cong R\nu_{X'/A',*}(\calE')$, with $\calE'$ the pull-back of $\calE$ to $(X'/A')_{\prism}$. Up to performing such a base change, we may assume that $I$ is principal.

	Moreover, 
 statements (1) and (2) are both local for the \'etale topology of $X$. Up to \'etale localization, we may assume that $X=\Spf(R)$ is affine and satisfies the assumptions in Situation~\ref{A:local-coordinate}. Then statement \ref{T:cohomology-Higgs}(1) follows immediately from Theorem~\ref{T:cohomology-Higgs}.

For (2), let $\widetilde R$ be a lift of $R$ as in Situation~\ref{A:local-coordinate}. For a morphism of  bounded prisms $(A,I)\to (A',I')$,  put $\widetilde R'=\widetilde R\widehat{\otimes}_AA'$,  $R'=\widetilde R'/I'\widetilde R'$ and $X'=\Spf(R')$. If $\calE$ corresponds to the Higgs module  $(\calE_{\tilde X}, \theta)$, then $\bar f_{\prism}^*\calE$ corresponds to $(\calE_{\tilde X}\otimes_{R}R', \theta\otimes 1)$.  By Theorem~\ref{T:cohomology-Higgs}, the canonical base change map 
	$f^{-1}R\nu_{X/A,*}(\calE)\otimes^L_{f^{-1}\cO_X}\cO_{X'}\to R\nu_{X'/A',*}(\bar f^*_{\prism}\calE)$ is identified in the derived category of $R'$-modules with the base change map
	\[
	\DR(\calE_{\tilde X},\theta)\otimes_RR'\xra{\sim} \DR(\calE_{\tilde X}\otimes_RR', \theta')
	\]
	which is clearly an isomorphism of complexes.

	\section{Poincar\'e Duality}
	\label{S:global application}
	We fix a bounded prism $(A,I)$, and a smooth $p$-adic formal scheme of relative dimension $n$ over $\Spf(A/I)$.

	\subsection{} Let $\CR((X/A)_{\prism},\overline{\cO}_{\prism})^{fr}$ be the category of $\overline{\cO}_{\prism}$-crystals locally free of finite rank.  We have the  natural notions of  tensor  product and internal hom in $\CR((X/A)_{\prism},\overline{\cO}_{\prism})^{fr}$:
	If $\calE_1,\calE_2$ are two objects of $\CR((X/A)_{\prism},\overline{\cO}_{\prism})^{fr}$, their tensor product   $\calE_1\otimes\calE_2$ is the $\overline{\cO}_{\prism}$-crystal such that  $$(\calE_1\otimes\calE_2)(U)=\calE_1(U)\otimes_{B}\calE_{2}(U)$$
	for every object $U=(\Spf(B)\leftarrow \Spf(B/IB)\to X)$ of $(X/A)_{\prism}$, and 
	$\underline{\hom}(\calE_1,\calE_2)$ is the $\overline{\cO}_{\prism}$-crystal  such that 
	\[
	\underline{\hom}(\calE_1,\calE_2)(U)=\Hom_{B}(\calE_1(U),\calE_2(U)).
	\]
	Assume  that  $X=\Spf(R)$ satisfies the assumptions in Situation~\ref{A:local-coordinate}. Let $(M_i,\theta_i)$ with $i=1,2$ be the object of $\Higgs^{\wedge}(R)$ corresponding to $\calE_i$ via the equivalence of categories in Theorem.~\ref{T:crystals-Higgs}. Then $\calE_1\otimes \calE_2$ corresponds to the Higgs module $(M_1\otimes_RM_2, \theta_1\otimes 1+1\otimes\theta_2)$, and $\underline{\hom}(\calE_1,\calE_2)$ corresponds to $(\Hom_R(M_1,M_2), \theta)$ such that 
	\[\theta(f)(x)=\theta_2(f(x))-(f\otimes1)(\theta_1(x))\in M_2\otimes_R \Omega^1_R\{-1\}\]
	for all $ x\in M_1.$

	\subsection{Duality pairing} Let $\calE$  be an object of $\CR((X/A)_{\prism},\overline{\cO}_{\prism})^{fr}$.  We put $\calE^\vee:=\underline{\hom}(\calE,\underline\cO_{\prism})$ and
	\[\calE\{i\}:=\calE\otimes_{A/I}(I/I^2)^{\otimes i}\]
	for all integers $i$. Then the cup product induces  a pairing 
	\begin{align}\label{E:crystal-pairing}
		R\nu_{X/A,*}(\calE^\vee\{n\})	\otimes^L_{\cO_X}R\nu_{X/A,*}(\calE) &\to R\nu_{X/A,*}(\calE\otimes\calE^\vee\{n\})\\
		&\to R\nu_{X/A,*}(\overline{\cO}_{\prism}\{n\})\ra \Omega^n_{X}[-n],\nonumber
	\end{align}
	where the last map is induced by the Hodge--Tate comparison isomorphism \cite[Theorem~4.11]{BS}
	\[R^n\nu_{X/A,*}(\overline{\cO}_{\prism})\cong \Omega^n_X\{-n\}.\]
	If moreover $X$ is proper over $\Spf(A/I)$, then one has similarly a pairing of perfect complexes of $A/I$-modules:
	\begin{equation}\label{E:cohomology-pairing}
		R\Gamma((X/A)_{\prism}, \calE^\vee\{n\})\otimes_{A/I}^L R\Gamma((X/A)_{\prism}, \calE)\to R\Gamma((X/A)_{\prism} \overline{\cO}_{\prism}\{n\})\to A/I[-2n],
	\end{equation}
	where the last map is induced by the Grothendieck--Serre trace map  
	\begin{equation}\label{E:HT-isomorphism}
		H^{2n}((X/A)_{\prism}, \overline{\cO}_{\prism}\{n\})\cong H^n(X_{\et}, R^n\nu_{X/A,*}(\overline{\cO}_{\prism}\{n\}))=H^n(X_{\et}, \Omega^n_X)\to  A/I.
	\end{equation}

	\begin{theorem}\label{T:poincare-pairing}
		Under the notation above, the following statements hold:
		\begin{enumerate}
			\item The pairing \eqref{E:crystal-pairing} induces an isomorphism of perfect complexes of $\cO_X$-modules: 
			\[R\nu_{X/A,*}(\calE^\vee\{n\})\xra{\sim}\RHom_{\cO_X}(R\nu_{X/A,*}(\calE),\Omega^n_X)[-n].\]
			
			\item If moreover $X$ is proper over $\Spf(A/I)$,  the pairing \eqref{E:cohomology-pairing} 
			induces an isomorphism of perfect complexes of $(A/I)$-modules
			\[R\Gamma((X/A)_{\prism}, \calE^\vee\{n\})\cong R\Hom_{A/I}(R\Gamma((X/A)_{\prism}, \calE),A/I)[-2n]\]
		\end{enumerate} 
	\end{theorem}
	
	\begin{proof}
		It is clear that statement (2) is an immediate consequence of (1) and the classical Grothendieck--Serre duality. Since statement (1) is local for the \'etale topology of $X$, up to \'etale localization we may assume thus that $X=\Spf(R)$ satisfies the assumptions in Situation~\ref{A:local-coordinate}. Let $(M,\theta)$ be the object of $\Higgs^{\wedge}(R)$ corresponding to $\calE$ by Theorem~\ref{T:crystals-Higgs}, and $(M^\vee\{n\}, \theta^\vee)$ be the object corresponding to $\calE^\vee\{n\}$. Then we have $M^\vee\{n\}=\Hom_R(M,R\{n\})$ and $\theta^\vee$ is given by 
		\[\theta^\vee(f)(x)=-(f\otimes 1)(\theta(x))\in \Omega^1_{R}\{n-1\},\]	
		for all $x\in M$ and $f\in M^\vee\{n\}$. By Theorem~\ref{T:cohomology-Higgs}, the pairing \eqref{E:crystal-pairing} is represented by the pairing of  complexes 
		\[\DR(M^\vee\{n\},\theta^\vee)\otimes_{R} \DR(M, \theta)\to \Omega^n_R[-n]\]
		given by 
		\[\langle f\otimes \omega_i,x\otimes\eta_j\rangle =\begin{cases}0 &\text{if $i+j\neq n$},\\
			f(x)\omega_i\wedge \eta_j &\text{if } i+j=n,
		\end{cases}\]
		for  $x\in M$, $f\in M^\vee\{n\}$, $\omega_i\in \Omega^i_R\{-i\}$  and $\eta_j\in \Omega^j_R\{-j\}$. Now, it is straightforward  to verify that, with the sign convention \S\ref{S:sign-covention} ,  such a pairing induces an isomorphism of complexes
		\[\DR(M^\vee\{n\}, \theta^\vee)\cong \Hom(\DR(M,\theta), \Omega^n_R)[-n].\]
		This finishes the proof of (1).
	\end{proof}

	\begin{remark}\label{R:poincare-duality}
		Assume that $X$ is proper and smooth of relative dimension $n$ over $\Spf(A/I)$.
		If one can construct a trace map 
		\[\mathrm{Tr}_X\colon H^{2n}((X/A)_{\prism}, \cO_{\prism}\{n\})\to A\]
		which reduces to the classical trace map  \eqref{E:HT-isomorphism} when modulo $I$ (so that $\mathrm{Tr}_X$ itself is an isomorphism if $X$ is geometrically connected), then Theorem~\ref{T:poincare-pairing} implies a similar duality for $\cO_{\prism}$-crystals.
	\end{remark}

\end{document}